\documentclass[reqno]{amsart}
\usepackage[colorlinks]{hyperref}
\usepackage{amssymb,amsmath,amsthm,subfig,graphicx,hypcap,todonotes,enumerate}
\usepackage[margin=1.25in]{geometry}
\usepackage[all]{xy}
\DeclareMathOperator{\crit}{crit}

\DeclareMathOperator{\im}{im}
\DeclareMathOperator{\Sp}{Span}
\DeclareMathOperator{\Kr}{Ker}
\DeclareMathOperator{\ind}{ind}

\DeclareMathOperator{\mcg}{MCG}
\DeclareMathOperator{\Sym}{Sym}
\DeclareMathOperator{\id}{id}

\newcommand{\co}{\colon\thinspace}

\newcommand{\C}{\mathbb{C}}
\newcommand{\R}{\mathbb{R}}
\newcommand{\Z}{\mathbb{Z}}

\newcommand{\x}{\mathbf{x}}
\newcommand{\y}{\mathbf{y}}
\newcommand{\z}{\mathbf{z}}
\newcommand{\w}{\mathbf{w}}

\newcommand{\cM}{\mathcal{M}}
\newcommand{\ba}{\boldsymbol{\alpha}}
\newcommand{\bb}{\boldsymbol{\beta}}
\newcommand{\bth}{\boldsymbol{\theta}}

\newcommand{\cD}{\mathcal{D}}
\newcommand{\cP}{\mathcal{P}}
\newcommand{\ff}{\mathfrak{f}}

\newcommand{\fg}{\mathfrak{g}}
\newcommand{\fH}{\mathfrak{H}}
\newcommand{\fm}{\mathfrak{m}}
\newcommand{\fn}{\mathfrak{n}}
\newcommand{\fo}{\mathfrak{o}}

\newcommand{\arx}{\vec{\mathbf{x}}}
\newcommand{\arz}{\vec{\mathbf{z}}}

\newcommand{\csf}{???}
\newcommand{\MYhref}[3][blue]{\href{#2}{\color{#1}{#3}}}

\theoremstyle{definition}
\newtheorem{thm}{Theorem}[section]
\newtheorem{lemma}[thm]{Lemma}
\newtheorem{ex}[thm]{Example}
\newtheorem{rmk}[thm]{Remark}
\newtheorem{prop}[thm]{Proposition}
\newtheorem{cor}[thm]{Corollary}
\newtheorem{df}[thm]{Definition}
\newtheorem{claim}[thm]{Claim}

\newtheorem{quest}[thm]{Question}
\newtheorem{sublemma}[thm]{Sublemma}
\title{Holomorphic polygons and smooth 4-manifold invariants}
\author{Jonathan Williams}
\begin{document}
\begin{abstract}
Any smooth, closed oriented 4-manifold has a surface diagram of arbitrarily high genus $g>2$ that specifies it up to diffeomorphism. The goal of this paper is to prove the following statement: For any smooth, closed oriented 4-manifold $M$, there is a sequence of weak $A$-infinity algebras indexed by $g$, and the homotopy equivalence class of each entry of this sequence is a diffeomorphism invariant of $M$.\end{abstract}
\maketitle
\tableofcontents
\begin{section}{Introduction}It is well known that generally applicable theories of holomorphic curve invariants have contributed much to the fields of symplectic topology and 3-manifolds, but such tools have not yet appeared for general smooth, closed oriented 4-manifolds. This paper applies ideas from \cite{AJ} and \cite{L} to surface diagrams, which are known to exist for any smooth, closed oriented $4$-manifold. On a basic level, applying a cylindrical Heegaard-Floer type of construction when the Lagrangians are immersed does not result in a homology group, but in a \emph{weak $A_\infty$ algebra}. This is a more complicated algebraic object than homology groups that can still record information about the relationship between domains and intersection points in a surface decorated with simple closed curves. The goal of this paper is to prove the following theorem, for which the necessary vocabulary follows. \begin{thm}\label{thm}For any smooth, closed oriented 4-manifold $M$ equipped with an admissible surface diagram of genus $g>2$, there is a weak $A_\infty$ algebra $(\cP,\fm)(M,g)$ over $\Z_2$. The homotopy equivalence class of $(\cP,\fm)(M,g)$ is a diffeomorphism invariant of $M$ for each $g$.\end{thm}
\begin{subsubsection}*{Acknowledgments}\ The author would like to thank Timothy Perutz, Michael Usher and Christopher Woodward for helpful conversations, and Robert Lipshitz for finding an important error in an earlier draft.\end{subsubsection}
\end{section}
\begin{section}{Basic definitions and notation}\label{basicdefs}
As in \cite{W2}, define the surface diagram $(\Sigma,\Gamma)$ for $M$ as one coming from a simplified purely wrinkled fibration $f\co M\to S^2$, where $\Gamma=\{\gamma_1,\ldots,\gamma_k\}$ is the $\Z/k\Z$-indexed collection of vanishing cycles in $\Sigma$ of $f$ (in this paper, $k$ is always $|\Gamma|$). Because every 4-manifold has a surface diagram coming from a map that is homotopic to a constant map $M\to p\in S^2$, and $H_2(M)$ is in some sense fully expressed by the surface diagram in that case (see Proposition~\ref{3handles}), we assume throughout that $f$ is homotopic to a constant map. Note, however, that the invariance results in this paper apply equally well to other homotopy classes of maps, so that the algebra $(\cP,\fm)$ in Theorem~\ref{thm} could be generalized to be an invariant of the triple $(M,g,[f])$, where $[f]$ is the homotopy class of a continuous map $f\co M\to S^2$. The relationship between algebras for varying $g$ and $[f]$ is currently unknown.

\begin{subsection}{Weak \emph{A}-infinity algebras}\begin{df}In this paper, a \emph{weak $A_\infty$ algebra} is a $\Z_2$ vector space $\cP$ and a sequence $\fm$ of $\Z_2$-multilinear maps $\fm_n\co\cP^{\times n}\to\cP$, defined for all $n\geq0$, satisfying the $A_\infty$ relation
\begin{equation}\label{eq:ainftyrlns}\sum\limits_{i+j=n+1}\sum\limits^{n-j+1}_{\ell=1}\fm_i\left(a_1,\ldots,a_{\ell-1},\fm_j\left(a_\ell,\ldots,a_{\ell+j-1}\right),a_{\ell+j},\ldots,a_n\right)=0\end{equation}
for each $n\geq0$ and each $n$-tuple of generators. Here, $\cP^{\times0}=1\in\Z_2$, so that $\fm_0\in\cP$.\end{df}

In the language of \cite{AJ}, $(\cP,\fm)$ is an ungraded \emph{weak} $A_\infty$ algebra over $\Z_2$. The term \emph{curved} $A_\infty$ algebra also sometimes refers to an $A_\infty$ algebra with $\fm_0\neq0$, but this author reserves that word for those whose maps $(\fm_n)_{n>0}$ are understood to have been deformed according to a nonzero $\fm_0$ as in \cite[Definition~3.19]{AJ}. We use the symbol $\cP$ for \emph{preliminary}. Because the moduli spaces seem likely to have coherent orientations and the Lagrangians are all orientable, the lack of grading and signs may soon be rectified with, for example, a suitable version of \cite[Definition~4.4]{AJ} or simply employing the straightforward $\Z_2$ grading on $\cP$ according to whether the intersections of Lagrangians that represent generators of $\cP$ are positive or negative. For this reason, at the very least, the author anticipates a moderate elaboration of the current work to produce a weak $A_\infty$ algebra with coefficients in $\Z$ rather than $\Z_2$. In later work the author may further enrich the structure to a \emph{gapped filtered} $A_\infty$ algebra; see \cite[Section~3.5]{AJ} for definitions.

Section~\ref{homotopypreliminaries} has a discussion of various relative gradings according to integers and homological data of $M$ that already make sense for $\fm_1$, and if it turns out that a particular algebra as constructed in this paper admits what is called a \emph{bounding cochain}, so that $\fm_1$ can be deformed into a differential, then perhaps these gradings will become useful (see \cite[Section~3.6]{AJ} for discussion of bounding cochains). For now, much of Section~\ref{homotopypreliminaries} serves merely as an initial attempt to connect the generators of $\cP$ and the curves counted by $\fm$ to the 4-manifold $M$.
\begin{df}A \emph{morphism of weak $A_\infty$ algebras} $\ff\co(\cP^1,\fm)\to(\cP^2,\fn)$ is a sequence of $\Z_2$-multilinear maps $(\ff_n)_{n\geq1}$, where each $\ff_n\co(\cP^1)^{\times n}\to\cP^2$ satisfies 
\begin{equation}\label{eq:morphismrlns}\begin{array}{c}
\sum\limits_{1\leq\ell\leq j\leq n}\ff_{n-j+\ell+1}\left(a_1,\ldots,a_{\ell-1},\fm_{j-\ell}(a_\ell,\ldots,a_{j-1}),a_j,\ldots,a_n\right)\\ = \sum\limits_{0<n_1<\cdots<n_\ell=n}\fn_\ell(\ff_{n_1}(a_1,\ldots,a_{n_1}),\ff_{n_2-n_1}(a_{n_1+1},\ldots,a_{n_2}),\ldots,\ff_{n_\ell-n_{\ell-1}}(a_{n_{\ell-1}+1},\ldots,a_{n_\ell})).\end{array}\end{equation}
For morphisms $\ff\co(\cP^1,\fm)\to(\cP^2,\fn)$ and $\fg\co(\cP^2,\fn)\to(\cP^3,\fo)$, the \emph{composition} $\fg\circ\ff\co(\cP^1,\fm)\to(\cP^3,\fo)$, which is associative, is given by
\begin{equation}\label{eq:comp}\begin{array}{rr}&(\fg\circ\ff)_n(a_1,\ldots,a_n)=\sum\limits_{0<k_1<\cdots<k_\ell=n}\fg_\ell(\ff_{k_1}(a_1,\ldots,a_{k_1}),\ff_{k_2-k_1}(a_{k_1+1},\ldots,a_{k_2}),\\&\ldots,\ff_{k_\ell-k_{\ell-1}}(a_{k_{\ell-1}+1},\ldots,a_{k_\ell})).\end{array}\end{equation}\end{df}
\begin{df}Let $\ff,\fg\co(\cP^1,\fm)\to(\cP^2,\fn)$ be morphisms of weak $A_\infty$ algebras. A \emph{homotopy from $\ff$ to $\fg$} is $\fH=(\fH_n)_{n\geq1}$, where each $\fH_n\co(\cP^1)^{\times n}\to\cP^2$ is a $\Z_2$-multilinear map satisfying
\begin{equation}\label{eq:htpyrlns}\begin{array}{c}\ff_n(a_1,\ldots,a_n)+\fg_n(a_1,\ldots,a_n)=\\ \  \\ \sum\limits_{\substack{0<j_1<j_2<\cdots<j_\ell<\\k_1<k_2<\cdots<k_m=n}}\begin{array}{l}
\fn_{\ell+m+1}(\ff_{j_1}(a_1,\ldots,a_{j_1}),\ff_{j_2-j_1}(a_{j_1+1},\ldots,a_{j_2}),\ldots,\\
\ff_{j_\ell-j_{\ell-1}}(a_{j_{\ell-1}},\ldots,a_{j_l}),\fH_{k_1-j_\ell}(a_{j_\ell+1},\ldots,a_{k_1}),\\
\fg_{k_2-k_1}(a_{k_1+1},\ldots,a_{k_2}),\ldots,\fg_{k_m-k_{m-1}}(a_{k_{m-1}+1},\ldots,a_{k_m}))\end{array}\\ \ \\
+\sum\limits_{i+j=n+1}\sum\limits^{n-j+1}_{\ell=1}\fH_i\left(a_1,\ldots,a_{\ell-1},\fm_{j-\ell}(a_\ell,\ldots,a_{j-1}),a_j,\ldots,a_n\right).\end{array}\end{equation}
We say \emph{$\ff$ is homotopic to $\fg$} when there is a homotopy $\fH$ from $\ff$ to $\fg$.\end{df}
\begin{df}The \emph{identity morphism} $\id_\cP\co(\cP,\fm)\to(\cP,\fm)$ is the morphism such that 
\[(\id_\cP)_n(a_1,\ldots,a_n)=\left\{\begin{array}{rc}a_1,&n=1\\0,&n>1\end{array}\right.\]\end{df}
\begin{df}For a morphism $\ff\co(\cP^1,\fm)\to(\cP^2,\fn)$ of weak $A_\infty$ algebras, a \emph{homotopy inverse} is a morphism $\fg\co(\cP^2,\fn)\to(\cP^1,\fm)$ such that $\fg\circ\ff$ and $\ff\circ\fg$ are both homotopic to the identity morphisms on $(\cP^1,\fm)$, $(\cP^2,\fn)$, respectively. The morphism $\ff$ is a \emph{homotopy equivalence} when $\ff$ has a homotopy inverse.\end{df}\end{subsection}

\begin{subsection}{Generators and moduli spaces}\label{spaces}In this paper, $\Sigma$ is understood to come with an embedding $\Sigma\hookrightarrow M$ as a fiber of $f$, so that for example there is a canonical induced map $H_\bullet(\Sigma)\to H_\bullet(M)$. The elements of $\Gamma$ are called attaching circles in the context of handlebody decompositions, or vanishing cycles in the context of stable maps to the sphere. In this paper, they are called \emph{circles}, and \emph{curves} refer to 
$J$-holomorphic curves. Choose orientations for the elements of $\Gamma$ once and for all and assume all circles intersect at right angles in $\Sigma$, for some chosen Riemannian metric on $\Sigma$. For each $\gamma_i\in\Gamma$ and for $t\in[0,1]$, let $P_i^t$ be a small perturbation generated by the vector field $X_i$ satisfying the following properties, where $\gamma_i^t$ is the result of flowing $\gamma_i$ according to $P_i^s,$ $s\in[0,t]$ and $\Gamma^t=\{\gamma_i^t\}_{i=1,\ldots,k}$.\begin{enumerate}
\item $X_i$ is supported near $\gamma_i$, and near all intersections with $\gamma_j$, $j\neq i$, $\Sigma$ has local coordinates in the $xy$ plane in which $X_i$ is given by $\partial/\partial y$ and $\gamma_i=\{y=0\}$.
\item For $0\leq t_1<t_2\leq1$, $\gamma_i^{t_1}$ intersects $\gamma_i^{t_2}$ at two transverse points.\end{enumerate}
Property (1) implies the $k$ perturbations $P_1,\ldots,P_k$ commute, so it makes sense to refer to them as a single perturbation. 
\begin{df}The collection $\left\{P_1,\ldots,P_k\right\}$ is called the \emph{main perturbation} of $\Gamma$.\end{df}
For any chosen pair $0\leq t_1<t_2\leq1$, a generator of $\cP$ is represented by a $k$-tuple of points $\x=\{x_1,\ldots,x_k\}$ where $x_i\in\gamma^{t_1}_i\cap\gamma^{t_2}_{s(i)}$ and $s$ is an element of the symmetric group on $k$ letters. This is similar to the definition of generators in \cite{L}, in that the two perturbations of $\Gamma$ are analogous to $\ba$ and $\bb$. For $0\leq t_1<t_2<t_3\leq1$, there is a correspondence between such $k$-tuples for the pairs $(t_1,t_2)$ and $(t_2,t_3)$ obtained by flowing each entry $x_i$ according to the perturbation $P_i^{t}$, $t\in[t_1,t_3]$, for each $i\in1,\ldots,k$, and relabeling it $x_{s(i)}$. In this perturbation, the intersection point travels along $\gamma_{s(i)}^{t_2}$ from an intersection with $\gamma_i^{t_1}$ to its nearest intersection with $\gamma_i^{t_3}$. Thus, the corresponding intersection point lies in $\gamma_{s(i)}^{t_2}\cap\gamma_i^{t_3}$, and, according to the notation chosen for $k$-tuples, the $\ell^{th}$ entry must lie in $\gamma_\ell^{t_2}\cap\gamma_{s(\ell)}^{t_3}$, hence the relabeling. This correspondence gives an equivalence relation between $k$-tuples for any choice of pairs $(t_n,t_{n+1})$ for which $n=1,\ldots,m$ and $0\leq t_n<t_{n+1}\leq1$, and a generator for the algebra $\cP$ is an equivalence class of such $k$-tuples. Here follows the definition of the spaces in which the holomorphic curves counted by $\fm$ will live.
\begin{df}For $n\geq1,$ let $W_{n+1}=\Sigma\times\Delta_{n+1}$, where $\Delta_{n+1}$ is conformally equivalent to the closed unit disk with $n+1$ boundary punctures: \[\left\{z\in\C:|z|\leq1\right\}\setminus\left\{e^{2\pi i\cdot\frac{1}{n+1}},e^{2\pi i\cdot\frac{2}{n+1}},\ldots,e^{2\pi i\cdot\frac{n+1}{n+1}}\right\}.\] Label each edge $e^{\ell/(n+1)}$ of $\Delta_{n+1}$ connecting $e^{2\pi i\cdot\frac{\ell-1}{n+1}}$ to $e^{2\pi i\cdot\frac{\ell}{n+1}}$ with its standard orientation in $\C$, so that $e^{1/(n+1)}$ is the first edge traveling counter-clockwise from 1 to $e^{2\pi i/(n+1)}$, and label the common limit point of $e^{\ell/(n+1)}$ and $e^{(\ell+1)/(n+1)}$ by $v^{\ell/(n+1)}$, so that $v^{(n+1)/(n+1)}=1\in\C$. When it is not otherwise specified, $W_{n+1}$ comes decorated with the $(n+1)k$ cylinders 
\[C_i^{\ell/(n+1)}=\gamma^\frac{\ell}{n+1}_i\times e^{\ell/(n+1)},\]
where $1\leq i\leq k$ and $1\leq\ell\leq n+1$.\end{df} 
\begin{figure}[h]\capstart\begin{center}\includegraphics{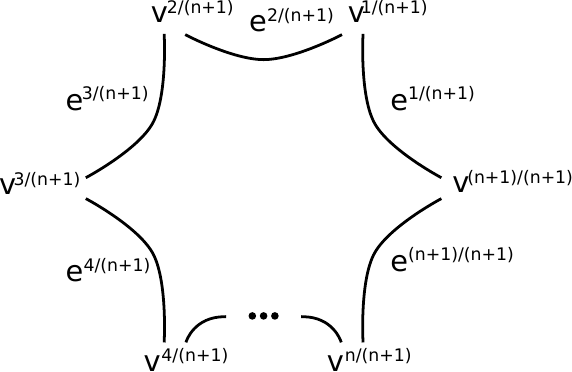}\end{center}\caption{Labeling of the disk $\Delta_{n+1}$ used in the definition of $\fm_n$ below.\label{polygon}}\end{figure}
Now it is time to give $W_{n+1}$ \emph{cylindrical ends} in the sense of \cite{L,BEHWZ}. Define the projections $\pi_{\Delta_{n+1}}\co W_{n+1}\to\Delta_{n+1}$, $\pi_\Sigma\co W_{n+1}\to\Sigma$, and $\pi_\R\co W_2\to\R$.
\begin{df}For a given $v^{\ell/(n+1)}$, fix a point $z_i$ in each component $D_i$ of $\Sigma\setminus\left(\Gamma^\frac{\ell}{n+1}\cup\Gamma^\frac{\ell+1}{n+1}\right)$. Let $j_\Sigma$ be a complex structure on $\Sigma$ tamed by an area form $dA$, and equip $[0,1]\times\R$, a strip that has coordinates $(s,t)$, with the area form $ds\wedge dt$. Further, let $\omega=ds\wedge dt+dA$, a split symplectic form on $\Sigma\times[0,1]\times\R$. Call an almost complex structure $J=J^{\ell/(n+1)}$ on $\Sigma\times[0,1]\times\R$ \emph{cylindrical} if it satisfies the following conditions.\begin{itemize}\item[({\bf 
J$^\ell$1})]\hypertarget{j1} $J$ is tamed by the split symplectic form $\omega$.\item[({\bf 
J$^\ell$2})]\hypertarget{j2} In a cylindrical neighborhood $U_i$ of $\{z_i\}\times[0,1]\times\R$, $J=j_\Sigma\times j_{[0,1]\times\R}$ is split. Here, $[0,1]\times\R$ is taken as a subset of $\C$ and $U_i$ is small enough that its closure does not intersect $\left(\Gamma^\frac{\ell}{n+1}\cup\Gamma^\frac{\ell+1}{n+1}\right)\times[0,1]\times\R$.\item[({\bf 
J$^\ell$3})]\hypertarget{j3} $J$ is translation invariant in the $\R$ factor.\item[({\bf 
J$^\ell$4})]\hypertarget{j4} $J(\partial/\partial s)=\partial/\partial t$.\end{itemize}\end{df} 
Let $B^{\ell/(n+1)}$ denote an open disk centered at $v^{\ell/(n+1)}$ with radius $\varepsilon_n<\frac{1}{4}|v^{\ell/(n+1)}-v^{(\ell+1)/(n+1)}|$ and let $\overline{W_{n+1}}$ denote the complement
\[\overline{W_{n+1}}=W_{n+1}\setminus\bigcup\limits_{\ell=1}^{n+1}\Sigma\times B^{\ell/(n+1)}.\]
\begin{df}Pick cylindrical $J^{1/(n+1)},\ldots,J^{(n+1)/(n+1)}$ and a point $z_i$ in each component of $\Sigma\setminus\bigcup_{\ell=1}^{n+1}\Gamma^\frac{\ell}{n+1}$. Let $\omega=dA_\Sigma+dA_{\Delta_{n+1}}$ be the split symplectic form on $W_{n+1}$. Call an almost complex structure $J=J_{n+1}$ on $W_{n+1}$ \emph{admissible} if it satisfies the following.\begin{itemize}\item[({\bf 
J1})]\hypertarget{J1} $J$ is tamed by the split symplectic form on $W_{n+1}$.
\item[({\bf 
J2})]\hypertarget{J2} In a neighborhood $U_i$ of $\{z_i\}\times\Delta_{n+1}$, $J=j_\Sigma\times j_{\Delta_{n+1}}$ is split. Here, $U_i$ is small enough so that its closure is disjoint from $\bigcup_{\ell=1}^{n+1}\Gamma^\frac{\ell}{n+1}\times\Delta_{n+1}$.
\item[({\bf 
J3})]\hypertarget{J3} Near each end $\Sigma\times\{v^{\ell/(n+1)}\}$, $J$ agrees with $J^{\ell/(n+1)}$. 
\item[({\bf 
J4})]\hypertarget{J4} The projection $\pi_{\Delta_{n+1}}$ is holomorphic and each fiber of $\pi_\Sigma$ is holomorphic.
\item[({\bf 
J5})]\hypertarget{J5} There is a 2-plane distribution $\xi$ on $\overline{W_{n+1}}$ such that $\omega|_\xi$ is nondegenerate, $J$ preserves $\xi$, and the restriction of $J$ to $\xi$ is compatible with $\omega$. In addition, $\xi$ is tangent to $\Sigma$ near $C_i^{\ell/(n+1)}\cap\overline{W_{n+1}}$ and $\overline{W_{n+1}}\cap\overline{B^{\ell/(n+1)}}$ for all $i,\ell$.\end{itemize}
\end{df}
The definitions above are analogues of those in \cite[Section~10.2]{L}, with (\hyperlink{j5}{{\bf J5}}) being the analogue of Lipshitz's condition ({\bf J$'$5}), for polygons. We immediately use that condition instead of Lipshitz's more standard condition ({\bf J5}) because something akin to the so-called \emph{annoying curves} in $\partial W_{n+1}$ will be an unavoidable part of the theory in this paper. From now on, we assume $W_{n+1}$ is equipped with cylindrical ends as specified by the product symplectic structure and admissible $J$.

Let $\mathcal{M}_{n+1}$ denote the moduli space of Riemann surfaces $S$ with boundary, $nk$ negative punctures $\mathbf{p}^\ell=\{p^\ell_1,\ldots,p^\ell_k\}$ for $1\leq\ell\leq n$ and $k$ positive punctures $\mathbf{q}=\{q_1,\ldots,q_k\}$, all on the boundary of $S$. The surface $S$ is compact away from the punctures and is considered modulo automorphism. For admissible $J_{n+1}$, this paper makes use of the space of holomorphic maps $u\co S\to W_{n+1}$ such that\begin{itemize}\item[({\bf 
M0})]\hypertarget{m0} The source $S$ is smooth.\item[({\bf 
M1})]\hypertarget{m1} $u(\partial S)\subset\bigcup\limits_{i,\ell}C^{\ell/(n+1)}_{i}$.\item[({\bf 
M2})]\hypertarget{m2} There are no components of $S$ on which $\pi_{\Delta_{n+1}}\circ u(S)$ is constant.\item[({\bf 
M3})]\hypertarget{m3} For each $(\ell,i)$, $u^{-1}\left(C^{\ell/(n+1)}_{i}\right)$ consists of exactly one boundary arc of $S$.\item[({\bf 
M4})]\hypertarget{m4} $\lim\limits_{w\to p^\ell_i}\pi_{\Delta_{n+1}}\circ u(w)=v^{\ell/(n+1)}$ and $\lim\limits_{w\to q_i}\pi_{\Delta_{n+1}}\circ u(w)=v^{(n+1)/(n+1)}=1$.\item[({\bf 
M5})]\hypertarget{m5} The \emph{energy} of $u$, defined in \cite[Section~5.3]{BEHWZ}, is finite.\item[({\bf 
M6})]\hypertarget{m6} $u$ is an embedding.\end{itemize}
It follows from \cite[Proposition~5.8]{BEHWZ} and the choice of orientations for the ends that near each $\mathbf{p}^\ell$, a curve satisfying (\hyperlink{m0}{{\bf M0}})-(\hyperlink{m6}{{\bf M6}}) converges exponentially in the $-\R$ direction of the cylindrical end at $v^{\ell/(n+1)}$ to $\x^{\ell/(n+1)}\times[0,1]$, and near $\mathbf{q}$ it converges exponentially in the $+\R$ direction of the cylindrical end at $v^{(n+1)/(n+1)}$ to $\y$, for some representatives $\x^{1/(n+1)},\ldots,\x^{n/(n+1)},\y$ of generators of $\cP$.
\begin{df}\label{vocab}\ \begin{itemize}
\item At various times it will be useful to consider the collections $\Gamma^{t_1}$, $\Gamma^{t_2}$ for some $0\leq t_1<t_2\leq1$. Denote these by $\bb$, $\ba$, respectively, and their elements $\beta_i$, $\alpha_i$, $i=1,\ldots,k$. 
\item Unless otherwise stated, the symbol $\arx$ always denotes an $n$-tuple of generators
\[\arx=\left(\x^{1/(n+1)},\ldots,\x^{n/(n+1)}\right).\] Denote the collection of homotopy classes of maps satisfying (\hyperlink{m0}{{\bf M0}})-(\hyperlink{m5}{{\bf M5}}) for fixed $\arx,\y$ by $\pi_2(\arx,\y)$.
\item Given a collection of circles $C\subset\Sigma$, a \emph{region} is a connected component of $\Sigma\setminus C$. 
\item The \emph{domain of $\varphi\in\pi_2(\arx,\y)$} is the linear combination of regions whose coefficients are given by counting multiplicities in $\pi_\Sigma(\varphi)$ with sign. In particular, the domain of $\varphi$ neglects any \emph{trivial disks}: components of $\varphi$ given by $\{p\}\times\Delta_2$ for some $p\in\Sigma$. A domain is always assumed to come from some $\varphi\in\pi_2(\arx,\y)$.
\item A \emph{periodic domain} $q$ is an element of $\pi_2(\x,\y)$ such that $\partial q$ is a sum of elements of $\ba$, $\bb$. Sometimes the difference of the domains of two elements of $\pi_2(\arx,\y)$ will be called periodic.
\item A domain is \emph{positive} if all of its coefficients are nonnegative. More generally, for domains $A$ and $B$, $A<B$ means $n_{z_i}(A)<n_{z_i}(B)$ for all $z_i$.
\item For $A\in\pi_2(\arx,\y)$, define $\mathcal{M}^A$ to be the moduli space of holomorphic maps $u\co S\to W_{n+1}$ in the homology class $A$ satisfying (\hyperlink{m0}{{\bf M0}})-(\hyperlink{m6}{{\bf M6}}), and define \[\mathcal{M}(\arx,\y)=\bigcup\limits_{A\in\pi_2(\arx,\y)}\mathcal{M}^A.\] For $n=1$, mod out by translation in $W_2$ by setting $\widehat{\cM}(\x,\y)=\cM(\x,\y)/\R$.\end{itemize}\end{df}
Note that this paper does not make use of a base point, so the notion of \emph{periodic domain} is slightly different from that of Heegaard-Floer theory. It seems at least plausible that this theory could be enriched by adding a basepoint, though the necessary invariance results would not (for example) be an obvious adaptation of  \cite[Figure~3]{OS}. Even though the definition of the algebra $(\cP,\fm)(M,g)$ will not appear until Section~\ref{algdef}, this seems important enough to be highlighted as a formal question.
\begin{quest}Choose a point $z$ in the complement of $\bigcup_{t\in[0,1]}\Gamma^t$ and consider the hat, $\pm$ and $\infty$ constructions of $(\cP,\fm)(M,g)$, analogous to those in \cite{L}, that could result from this choice.\begin{enumerate}[(a)]
\item Would such a construction yield a smooth 4-manifold invariant?
\item What would be the geometric significance of the resulting gradings?\end{enumerate}\end{quest}
There is reason to be suspicious that, without a base point, our algebras are determined by homological data specified by the surface diagram: Heegaard-Floer homology is known to have such a property in various examples. On the other hand, the maps $\fm_n$ for $n>1$ could encode information about how that data is expressed in the surface diagram that is not detected by $m_1$. This is a subject of current research by the author.

There is one more collection of moduli spaces that shows up in the compactness Theorem~\ref{compactness}. For given $v^{\ell/(n+1)}$, consider $\Delta_1$ as the upper half-plane and take $W^{\ell/(n+1)}_1=\Sigma\times\Delta_1$, decorated with $k$ cylinders $C_i^{v^{\ell/(n+1)}}\subset\Sigma\times\R$, obtained by choosing cylinders that agree with $C^{\ell/(n+1)}$ near $-\infty$ and $C^{(\ell+1)/(n+1)}$ near $+\infty$. Choose a point $z_i$ in each component of $\Sigma\setminus\pi_\Sigma\left(\left\{C_i^{\ell/(n+1)}\right\}_{i=1}^k\right)$ and a complex structure $J$ satisfying (\hyperlink{j1}{{\bf J1}})-(\hyperlink{j5}{{\bf J5}}). Then for a generator $\y$ represented by a $k$-tuple of points in $\Gamma^\frac{\ell}{n+1}\cap\Gamma^\frac{\ell+1}{n+1}$, let $\pi_2(\y)$ be the collection of homotopy classes of maps $u\co S\to W^{\ell/(n+1)}_1$ satisfying (\hyperlink{m0}{{\bf M0}})-(\hyperlink{m5}{{\bf M5}}). Let $\cM^{\ell/(n+1)}(\y)$ denote the collection of rigid holomorphic embeddings representing elements of $\pi_2(\y)$. Since varying the superscript in this symbol corresponds to an arbitrarily small perturbation of $J$, which corresponds to a trivial cobordism of moduli spaces when $J$ achieves transversality, the superscript is only included when discussing a particular end of $W_{n+1}$, and we usually write $\cM(\y)$ to denote this moduli space.\end{subsection}
\end{section}
\begin{section}{Homotopy preliminaries}\label{homotopypreliminaries}
The section begins with a few topological constructions that are used throughout the rest of the paper. Sections \ref{sectors} and \ref{secttraj} are mainly meant as an introduction to the algebraic topology of surface diagrams and are not required for the proof of Theorem~\ref{thm}: Section~\ref{sectors} explains how each generator falls into a partition relatively indexed by $H_1(M)$ and further into a partition indexed by $H_1(X)$ called a \emph{sector}, and Section~\ref{secttraj} is a discussion of how the curves counted by $\fm$ relate to $H_2(M)$. Section~\ref{sdlemmas} has some crucial lemmas about surface diagrams that explain why they are manageable enough for the proposed 4-manifold invariants to be feasible. We go into some detail in this section because of the relative lack of surface diagram experts (compared to those who would more readily understand the arguments and constructions in the Floer theoretic parts of the paper).
\begin{subsection}{A certain submanifold}\label{handledecomp} \begin{figure}\capstart\begin{center}\includegraphics{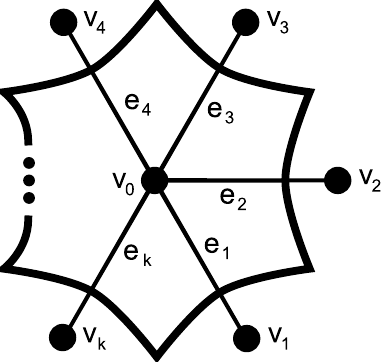}\end{center}\caption{A graph $G_k\subset S^2$, with the critical image of the fibration map in bold.\label{graph}}\end{figure}  $G_k$ be the graph in $S^2$ as in Figure~\ref{graph}, with one edge for every element of  $\Gamma$. Choose a regular neighborhood $\nu G_k$ so that its intersection with the critical image $f(\crit f)$ is a collection of $k$ arcs that contract to $G_k\cap\crit f$. Orienting the edges of $G_k$ to start at $v_0$, and labeling $G_k$ as in the figure, the preimage of each edge $e_n$ is a cobordism from $\Sigma_g$ to $\Sigma_{g-1}$ obtained by attaching a 3-dimensional 2-handle to the product of $\Sigma_g=f^{-1}(v_0)$ with $[0,1]$ along $\gamma_n\times\{1\}$. The preimage $X:=f^{-1}(\nu G_k)$ can then be constructed as a 4-dimensional handlebody by adding fiber-framed 2-handles to the boundary of $\Sigma_g\times D^2$ along regular neighborhoods of $\gamma_n\times e^{2\pi in/k}$, $n=1,\ldots,k$ (see, for example, \cite[Section~3.2]{Be}). The handles comprising the closure of $M\setminus X$ consist of a single 2-handle attached to a circle in $\partial X$ (determined up to isotopy in $\partial X$ by $\Gamma$), $2g-2$ 3-handles and a single 4-handle. Recall that the placement of the 3-handles and 4-handles in any smooth, closed 4-manifold is unique up to isotopy \cite[Section~4.4]{GS}. The placement of the 2-handle in $M\setminus X$ is determined up to isotopy in $\partial(M\setminus X)\cong\Sigma_{g-1}\times S^1$ by $\Gamma$ when $g>2$; see for example \cite[Corollary~2]{W1}.\end{subsection}
\begin{subsection}{Generators, sectors and first homology classes}\label{sectors}The following is an explanation of how generators are partitioned according to $H_1(M)$, and of how each partition is further partitioned into sectors.
\begin{subsubsection}{Generators and sectors} For two generators $\x,\y$, let $a_i$ be any 1-chain in $\alpha_i$ going from the entry of $\x$ to the entry of $\y$ in $\alpha_i$, and define $b_i$ similarly as a 1-chain in $\beta_i$. Then $c_{\x\y}=\sum_i(a_i-b_i)$ is a cycle in $\Sigma$, and define $\varepsilon(\x,\y)$ to be the image of this cycle under the map \[H_1(\Sigma)\to\frac{H_1(\Sigma)}{\iota_*\left(\Z^{|\ba|}+\Z^{|\bb|}\right)}\cong H_1(X),\]where $\iota_*$ is the inclusion homomorphism that sends $n\cdot\alpha_i+m\cdot\beta_j$ to the class $n[\alpha_i]+ m[\beta_j]$.\begin{prop}\label{sectorprop1}The set $\pi_2(\x,\y)$ is nonempty if and only if $\varepsilon(\x,\y)=0$.\end{prop}
\begin{proof}The argument is the same as the proof of equivalence of (1) and (2) in \cite[Lemma~2.2]{L}.\end{proof}
Proposition~\ref{sectorprop1} justifies the presence of a partition of the generators, relatively indexed by elements of $H_1(X)$.\begin{df}Each element of the partition of the generators according to $H_1(X)$ is called a \emph{sector}, and the notation $\x\in\sigma\in a\in H_1(X)$ means $\x$ is an element of the partition given by $\sigma$, and $\sigma$ is a representative of $a$.\end{df}There is an obvious generalization of $\varepsilon$ and Proposition~\ref{sectorprop1} to $\pi_2(\arx,\y)$, partitioning the generators of the vector space $\cP^{\otimes(n+1)}$. We use a tensor here instead of direct product because $\pi_2(\arx,\y)$ is empty if any of the entries of $\arx$ are the empty generator.\end{subsubsection}

\begin{subsubsection}{Sectors and $H_1(M)$}\label{secth1m}The purpose of this section is to delineate how the sectors of $\cP$ relate to each other and how the homotopy classes of curves counted by $\fm$ relate to $H_2(M)$. It turns out that sectors naturally represent elements of $H_1(M)$. Proposition~\ref{sectorprop3} is the main ingredient in the proof of this fact, by showing that every surface diagram has generators that represent $0\in H_1(M,\Z)$ in a sense that is described as follows.

In any surface diagram with $k$ attaching circles coming from a map $f\co M\to S^2$, the algebraic intersection number $\#(\gamma_i\cap\gamma_{i+1})$ is $\pm 1$, so it is always possible to choose a point $p_i\in\gamma_i\cap\gamma_{i+1}$ for each $i\in\Z/k\Z$. Let $Z_i$ be an arc in $\gamma_i$ such that $\partial Z_i=p_i-p_{i-1}$, and let $Z=\sum_iZ_i$.
\begin{prop}\label{sectorprop2}The circle $Z\subset\Sigma\subset M$ is freely isotopic to the critical circle of the fibration map. In particular, its free isotopy class is independent of the choice of point $p_i\in\gamma_i\cap\gamma_{i+1}$ and arc $Z_i$ for each $i\in\Z/k\Z$.\end{prop}
\begin{proof}Let $\crit(f)_i$ denote the closure of the fold arc whose vanishing cycle is $\gamma_i$. Parameterize $Z_i$ and $\crit(f)_i$ by $[0,1]$. For each $t\in(0,1)$, the local model for a fold gives a disk $D_i(t)\subset M$ with boundary $\gamma_i$, such that $Z_i(t)$ is connected to $\crit(f)_i(t)$ by a radius. For $t=0$ and $t=1$, the disks $D_i(t)$ exist by continuity of $f$, and $D_i(0)\cap D_{i-1}(1)$ is a radius of each disk connecting a cusp point to $p_i$. Now the proposition follows from the fact that $Z\subset\partial D$ and \[D=\bigcup\limits_{t\in[0,1],\ i\in\Z/k\Z}D_i(t)\] contracts to $\crit f$.\end{proof}
\begin{prop}\label{sectorprop3}The circle $Z$ represents $0\in H_1(M,\Z)$. \end{prop}
\begin{proof}The following proof is due to David Gay. Let $f\co M\to S^2$ be a simplified purely wrinkled fibration, and let $\nu Z$ be a tubular neighborhood of its critical circle $Z$. Fix an arbitrary choice of almost complex structure $j_\mathcal{H}$ on some horizontal distribution $\mathcal{H}\subset T(M\setminus\nu Z)$, and  choose a smoothly varying complex structure $j_a$ on the fibers $F_a$ of the map $f|_{M\setminus\nu Z}$, so that $J_0=j_\mathcal{H}\oplus j_a$ is an almost complex structure on $M\setminus Z$ with induced Spin$^c$ structure $\mathfrak{s}_0$. The first Chern class of $J_0$ evaluated on $[F_a]\in H_2(M\setminus Z)$ is $\chi(F_a)$, so that the values of this function taken at fibers on the higher and lower genus sides of $f(Z)$ differ by 2. Now choose an arbitrary Spin$^c$ structure $\mathfrak{s}$ on $M$. Its first Chern class gives a constant function of the fibers of $f$, so that $\mathfrak{s}|_{M\setminus\nu Z}$ is distinct from $\mathfrak{s}_0$, and differs from $\mathfrak{s}_0$ by the action of some nonzero $\alpha\in H^2(M\setminus\nu Z)$, Poincare dual to a nonzero $A\in H_2(M,Z)$. Now, work of Taubes \cite{T} shows that such classes are exactly the ones whose image under $\partial\co H_2(M,Z)\to H_1(Z)$ is $[Z]$ for an appropriate orientation of $Z$ (see also \cite[Section~4.2]{P}). Thus, the boundary of any representative of $A$ is homologous to $Z$, completing the proof. A way to see this more directly is to observe that\[\langle c_1(\mathfrak{s}_0)-c_1(\mathfrak{s}|_{M\setminus \nu Z}),[F_a]\rangle=2A\cdot[F_a]=\chi(F_a)-const.\] decreases by 2 when $a$ crosses from the higher-genus side to the lower-genus side of $f(Z)$, so that any representative of $A$ must have boundary homologous to $Z$.\end{proof}
\begin{figure}[h]\capstart\begin{center}\includegraphics{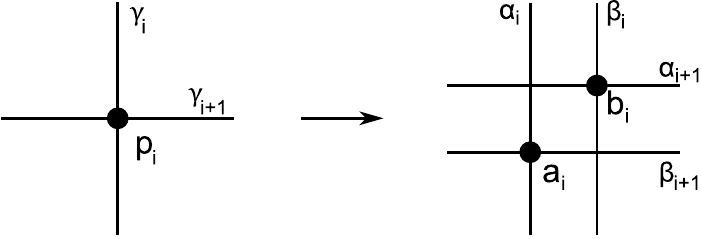}\end{center}\caption{Two generators result from a choice of $k$ points $p_i\in\gamma_i\cap\gamma_{i+1}$.\label{prime}}\end{figure}In contrast to spin$^c$ structures on closed 3-manifolds, there is a canonical sector for each surface diagram. Assuming there are no canceling intersections between $\gamma_i$ and $\gamma_{i+1}$, $i\in\Z/k\Z$, there are exactly two generators with entries chosen from $\{a_i,b_i:i=1,\ldots,k\}$ as in Figure~\ref{prime}, because choosing one entry from this set uniquely determines the other $k-1$ entries according to the requirement that each circle contains exactly one entry. Let $\x_0$ denote a choice, once and for all, of one of these generators.
\begin{df}Any generator resulting from a construction like the one producing $\x_0$ is called a \emph{reference generator}. For any generator $\x$, the element $\varepsilon(\x,\x_0)\in H_1(X)$ is called the \emph{sector} of $\x$ with respect to $\x_0$, denoted $\sigma_{\x_0}(\x)$. When there is no ambiguity about $\x_0$, the subscript is omitted.\end{df}
Recall the construction of the 1-cycle $c_{\x\y}\subset\Sigma\subset M$ at the very beginning of Section~\ref{sectors}.
\begin{prop}\label{sectorprop4}Let $\x_0$, $\x_0'$ be reference generators for some fixed $(\Sigma,\ba,\bb)$. For any generator $\x$, $[c_{\x\x_0}]=[c_{\x\x_0'}]\in H_1(M)$.\end{prop}
\begin{proof}Define $c_{\alpha_i}$ to be any 1-chain in $\alpha_i$ from the entry of $\x$ in $\alpha_i$ to the entry of $\x_0$ in $\alpha_i$, and let $c_{\alpha_i}'$ to be any 1-chain in $\alpha_i$ from the entry of $\x$ in $\alpha_i$ to the entry of $\x_0'$ in $\alpha_i$, and similarly define arcs $c_{\beta_i},c_{\beta_i}'$. Then on the chain level \[\begin{array}{ccc}c_{\x\x_0'}-c_{\x\x_0}&=&\sum_i(c_{\alpha_i}'-c_{\beta_i}')-\sum_i(c_{\alpha_i}-c_{\beta_i})\\&=&\sum_i(c_{\alpha_i}'-c_{\alpha_i})-\sum_i(c_{\beta_i}-c_{\beta_i}'),\end{array}\]and it is not hard to use Proposition~\ref{sectorprop2} to verify that the last two sums are oppositely oriented representatives of $\pm[\crit f]\in H_2(M)$, so that $[c_{\x\x_0'}-c_{\x\x_0}]=\pm2[\crit f]=0\in H_1(M)$ by Proposition~\ref{sectorprop3}.\end{proof}
\begin{cor}\label{topologicalpartitioncor}The assignments $\x\mapsto\sigma_{\x_0}(\x)\in H_1(X)$ and $\x\mapsto\sigma_{\x_0}(\x)\mapsto a\in H_1(M)$, where the last map is induced by the inclusion $X\hookrightarrow M$, do not depend on $\x_0$.\end{cor}
\begin{proof}The independence of the first assignment follows from the fact that $\pi_2(\x_0,\x_0')$ is nonempty for any pair of reference generators $\x_0,\x_0'$: Possibly after an isotopy, for a surface diagram $(\Sigma,\Gamma)$ assume that consecutive elements of $\Gamma$ have unique intersection. Then there are only two reference generators, and it is easy to construct an element of $\pi_2(\x_0,\x_0')$ by adding appropriate regions that lie between $\alpha_i$ and $\beta_i$ for $i=1,\ldots,k$. Reversing the isotopy of $\Gamma$ so that canceling intersections may appear between elements of $\Gamma$, any newly introduced reference generator is connected to $\x_0$ by adding appropriately signed bigons. Independence in the second assignment follows from the first and Proposition~\ref{sectorprop4}.\end{proof}
\begin{rmk}In Section~\ref{slides}, elements of $\Gamma$ are allowed to undergo a sequence of modifications in which they slide over each other much like handleslide moves in Heegaard diagrams, possibly creating more reference generators and resulting in pairs $(\Sigma,\Gamma^H)$ that are not surface diagrams, but for which the weak $A_\infty$ algebra $(\cP,\fm)$ is still defined. This presents no issue, because ``sectors" for the resulting diagrams (defined as equivalence classes of generators given by fibers of  $\varepsilon$) are calculated using the same reference generator as before the slide sequence, and after the sequence, the resulting triple $(\Sigma,\ba^s,\bb^s)$ always comes from a surface diagram, so that the argument above applies.\end{rmk}
With this understood, it makes sense to say that generators naturally fall into equivalence classes indexed by sectors, which themselves fall into equivalence classes indexed by $H_1(M)$. However, given $\x,\y\in a\in H_1(M)$ such that $\varepsilon(\x,\y)\neq0$, there does not seem to be any preferred difference in grading, because the relative Maslov index is defined as the index of any element of $\pi_2(\x,\y)$, which by Proposition~\ref{sectorprop1} is empty. For this reason, as is familiar from Heegaard-Floer theory, the relative Maslov grading (as defined) only gives a relative grading between generators in the same sector. However, considering the handlebody decomposition of Section~\ref{handledecomp}, there is one further ``section circle" $\gamma\subset\Sigma\subset M$ which is isotopic to the attaching circle of the 2-handle whose core transversely intersects each lower-genus fiber once (analogous to the 2-handle dual to the fiber 2-handle in Lefschetz fibrations); see for example \cite{B,GS} for further details. For $\x,\y\in a\in H_1(M)$, $\varepsilon$ takes values in the cyclic subgroup $G<H_1(X)$ generated by $[\gamma]$, because $H_1(M)\cong H_1(X)/G$. For this reason, the set of generators in sectors representing $a\in H_1(M)$ inherits a relative grading by integers according to multiples of $[\gamma]\in H_1(X)$.\end{subsubsection}\end{subsection}
\begin{subsection}{Relating trajectories to the 4-manifold}\label{secttraj}\label{trajectories}Here follows an explanation of the relationship between $\pi_2(\x,\y)$ and $H_2(M)$. There are versions of this material for $\pi_2(\arx,\y)$, but these are straightforward generalizations and we stick to $\pi_2(\x,\y)$ for the sake of exposition. Roughly speaking, $\pi_2(\x,\x)$ is a subgroup of the 2-chain group of $M$ corresponding to the handle decomposition given in Section~\ref{handledecomp}. When $[\Sigma]=0\in H_2(M)$ (as is the case for the surface diagrams we consider, which come from nullhomotopic maps), this subgroup is large enough to generate $H_2(M)$, so $\pi_2(\x,\x)$ could be considered a typically large extension of $H_2(M)$. Consider the following part of the long exact sequence for the pair $(M,X)$. \begin{equation}\label{eq:exseq1}\begin{array}{ccccccc}H_3(M,X)&\xrightarrow{\partial_3}&H_2(X)&\xrightarrow{i_{2\ast}}&H_2(M)&\xrightarrow{j_{2\ast}}&H_2(M,X)
\end{array}\end{equation}Though the term \emph{surface diagram} in this paper typically means a diagram coming from a nullhomotopic map, the following proposition applies to maps in any homotopy class.
\begin{prop}\label{3handles}\ \begin{enumerate}\item The map $j_{2\ast}\co H_2(M)\rightarrow H_2(M,X)$ is the zero map when $[\Sigma]=0\in H_2(M)$ so that $H_2(X)/\im\partial_3\cong H_2(M)$.
\item For a diagram $(\Sigma,\Gamma)$, when $[\Sigma]=0\in H_2(M)$ there is a commutative diagram\begin{equation}\label{eq:c3}\xymatrix{C_3(M,X)\ar[r]^{\phi}\ar[d]^{\iota_3}&C_2(X)\ar[d]^{\iota_2}\\ C_3(M)\ar[r]^{\partial} & C_2(M)}\end{equation}\end{enumerate}in which $\iota_3$ is an isomorphism induced by $\overline{M\setminus X}\hookrightarrow M$, $\iota_2$ is an inclusion $\Z^{2k}\hookrightarrow\Z^{2k+1}$ induced by $X\hookrightarrow M$, and $\phi$ is the chain map inducing $\partial_3$ in the sequence~(\ref{eq:exseq1}).\end{prop}
\begin{proof}For the first claim, observe that the closure $Y$ of $M\setminus X$ is diffeomorphic to $\Sigma'\times D^2$, where $\Sigma'$ is any lower-genus regular fiber. This can be seen by examining the restriction of the fibration map $f|_Y$, which appears as in 
\begin{figure}[h]\capstart\begin{center}\includegraphics{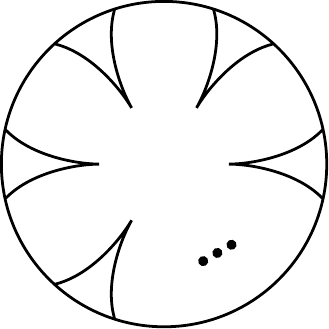}\end{center}\caption{The image of $Y$ under the fibration map. The regular fiber over the central region is $\Sigma'$, and otherwise the regular fiber is $\Sigma$.\label{lgs}}\end{figure}
Figure~\ref{lgs}. In that figure, let $S_r$ be the circle of radius $r\in(0,1]$ parallel to the boundary of the disk, which has radius 1. Then $\{f_r=f|_{f^{-1}(S_r)}:r\in(0,1]\}$ is a family of Morse functions $\Sigma'\times S^1\to S_r$ that is the obvious projection for $r$ small, and according to the local model for cusps, $f_r$ gains a canceling pair of index 1 and 2 critical points each time $S_r$ passes a cusp point. By excision, \[H_2(M,X)\cong H_2\left(\Sigma'\times D^2,\partial(\Sigma'\times D^2)\right)\cong\Z[\{pt\}\times D^2].\]In the cellular 2-chain group coming from the handle decomposition of $M$ in Section~\ref{handledecomp}, let $\gamma$ denote the generator corresponding to $\{pt\}\times D^2\subset Y$, so that $C_2(M)$ can be written $\Z^{2k+1}=\Sp_\Z(\gamma_1,\ldots,\gamma_k,\gamma)$. Any 2-cycle representing a nonzero element of $H_2(M)$ must intersect $\Sigma'$ algebraically zero times, since $[\Sigma]=[\Sigma']=0\in H_2(M)$. Since $\gamma$ is the only basis element that intersects $\Sigma'$, any cycle representing a nonzero homology element must have $\gamma$ appearing with coefficient zero. Thus the image of $H_2(M)$ in $H_2(M,X)$ is zero.

For the second claim, let the four chain groups in diagram~(\ref{eq:c3}) be the cellular 2-chain groups coming from the handle decomposition of $M$ in Section~\ref{handledecomp}. Then the vertical maps in diagram~(\ref{eq:c3}) are the obvious inclusions and it is not hard to see that $\phi$ is induced by the attaching maps for the 3-cells in $M$.\end{proof}
Because $X$ has no 3-handles, there is an isomorphism $\Kr\iota_\Gamma\rightarrow H_2(X)$, where $\iota_\Gamma\co\Z^k\to H_1(\Sigma)$ is induced by including each element of $\Gamma$. At the chain level, this isomorphism is obtained by mapping a surface $F$ into $\Sigma=f^{-1}(v_0)$ such that $\partial F=a\in\Kr\iota$, then capping its boundary circles with the ``thimbles" given by the corresponding folds in $X$. In a similar vein, for any generator $\x$ the projection $W_{\alpha\beta}\to\Sigma$ leads to an isomorphism $\pi_2(\x,\x)\cong\Kr\iota_{\alpha\beta}$, where $\iota_{\alpha\beta}\co\Z^{2k}\to H_1(\Sigma)$ comes from including the elements of $\ba$ and $\bb$. To relate $\pi_2(\x,\x)$ to $H_2(X)$, view $\pi_2(\x,\x)$ as a subgroup of the free Abelian group generated by the path components of $\Sigma\setminus(\ba\cup\bb)$ and define the homomorphism $p\co\pi_2(\x,\x)\to H_2(X)$ given by projecting out the coordinates of the regions that lie between $\alpha_i$ and $\beta_i$, $i=1,\ldots,k$. This new vector can be naturally interpreted as a 2-chain in $\Sigma\setminus\Gamma$ in the obvious way, and it is not difficult to verify that its boundary is a sum of the elements of $\Gamma$, so that capping it off with thimbles defines a cycle in $X$. Then for $\varphi\in\pi_2(\x,\x)$, $p(\varphi)$ is the class of that cycle.\\
The subgroup of $\pi_2(\x,\x)$ consisting of linear combinations of the regions lying between $\alpha_i$ and $\beta_i$ for $i=1,\ldots,k$ is called $\mathfrak{T}_{\alpha\beta}(\x,\x)$ in the following proposition and Section~\ref{slides}; see for example Definition~\ref{thin}.
\begin{prop}\label{h2prop}The homomorphism $p\co\pi_2(\x,\x)\to H_2(X)$ is surjective, and $\Kr p=\mathfrak{T}_{\alpha\beta}(\x,\x)$.\end{prop}
\begin{proof}It is clear that $\mathfrak{T}_{\alpha\beta}(\x,\x)\subset\Kr p$ by definition. Let $D$ denote the collection of path components of $\Sigma\setminus(\ba\cup\bb)$ that lie between $\alpha_i$ and $\beta_i$, $i=1,\ldots,k$. If $\varphi\in\pi_2(\x,\x)\setminus\mathfrak{T}_{\alpha\beta}(\x,\x)$, then there is at least one region in $\Sigma\setminus(\ba\cup\bb\cup D)$ with nonzero coefficient, so that $p(\varphi)\neq\boldsymbol{0}$. Thus, $\mathfrak{T}_{\alpha\beta}(\x,\x)\supset\Kr p$. For $A\in H_2(X)$,  it is straightforward to construct an element  $D\in p^{-1}(A)$ as a linear combination of path components of $\Sigma\setminus(\ba\cup\bb)$ by first marking each path component (other than the generators of $\mathfrak{T}_{\alpha\beta}(\x,\x)$) with the multiplicity specified by $A$, then choosing an element of $\mathfrak{T}_{\alpha\beta}(\x,\x)$ such that $\partial D$ is a linear combination of $\alpha$ and $\beta$ circles.\end{proof}\end{subsection}
\begin{subsection}{Lemmas about surface diagrams}\label{sdlemmas}Switching gears entirely and using terminology and ideas from \cite{W2}, \cite{B} and \cite{BH}, the goal of this section is to prove some results about surface diagrams that elaborate on the main result of \cite{W2} and appear in the proof of Theorem~\ref{thm}.
\begin{lemma}\label{slidelem}The handleslide and multislide moves are sequences of slides. The shift move is a sequence of slides, followed by a reordering of the elements of $\Gamma$.\end{lemma}
\begin{proof}In \cite[Section~2.2]{H} there appears a \emph{surgery homomorphism} \[\Phi_c\co\mcg(\Sigma)(c)\to\mcg(\Sigma'),\] where $\mcg(\Sigma)(c)$ is the mapping class group of orientation-preserving diffeomorphisms of the genus $g$ oriented closed surface $\Sigma$ that fix the unoriented isotopy class of the embedded circle $c$, and $\mcg(\Sigma')$ is the mapping class group of the genus $g-1$ surface $\Sigma'$ (see also the introduction of \cite[Section~2.3]{B}, where the map is called $\psi_\gamma$). According to Propositions 4.4 and 4.9 of \cite{BH}, any multislide or handleslide, applied using the pair $\gamma_1,\gamma_n\in\Gamma$, is realized by applying an element of $\Kr\Phi_{\gamma_1}\cap\Kr\Phi_{\gamma_n}$ to a subset of $\Gamma$.\footnote{See also \cite[Figure~6]{W2} for an example of how to construct an element of $\Kr\Phi_{\gamma_1}\cap\Kr\Phi_{\gamma_n}$ for the handleslide move, and how such an element is realized by a sequence of slides.}

For handleslides, $\gamma_1$ and $\gamma_n$ are disjoint and, according to \cite[Theorem~1.1]{W2}, we can assume their union is nonseparating. Consider the genus $g-2$ surface $\Sigma''$ obtained from replacing $\gamma_1$ and $\gamma_n$ with pairs of disks whose centers are labeled $v_1,v_2$ for the disks coming from $\gamma_1$ and $w_1,w_2$ for the disks coming from $\gamma_n$. Hayano constructs an element of $\Kr\Phi_{\gamma_1}\cap\Kr\Phi_{\gamma_n}$ in \cite[Section~3]{H} as the lift (according to the pair of surgeries specified by the pair of 0-spheres $(v_1,v_2)$ and $(w_1,w_2)$) to $\mcg(\Sigma)$ of the point pushing map defined by an oriented embedded arc $\eta\subset\Sigma''$ connecting $v_i$ to $w_j$, for $i,j\in\{1,2\}$. With this lift understood, it is easy to see that the lemma is proved for multislides. However, there is another way to see it: In \cite[Lemma~3.8]{H}, Hayano exhibits this lift as the product of Dehn twists $t_{\tilde\delta(\eta)}\cdot t_{\gamma_1}^{-1}\cdot t_{\gamma_n}^{-1}$, where $\tilde\delta(\eta)$ is the obvious lift of the embedded circle $\partial\overline{\nu\eta}$ to $\Sigma$, where $\nu\eta$ is a regular neighborhood of $\eta$. One can then deduce the lemma by checking it for each of the two ways an embedded circle can intersect the\begin{figure}\capstart\begin{center}\includegraphics{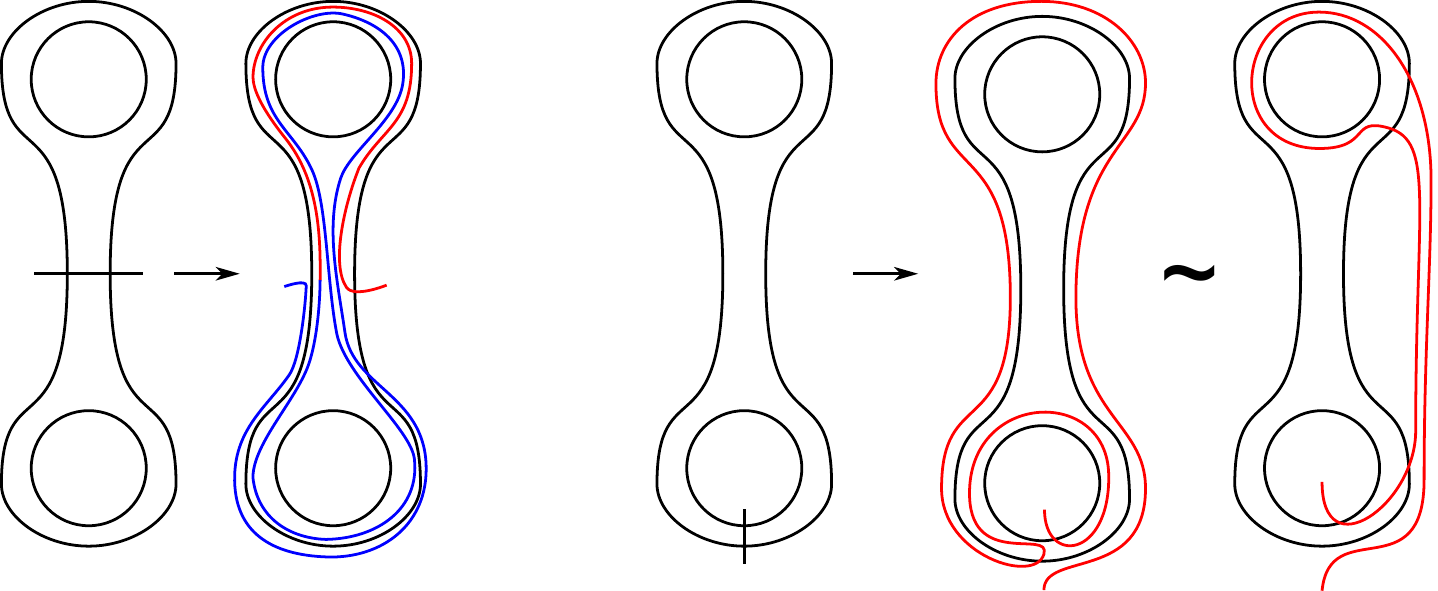}\end{center}\caption{Checking that the surface diagram handleslide move is a sequence of slides. The two smaller circles are interchangeably $\gamma_1$ and $\gamma_n$ along each of which the arc gets a negative Dehn twist, while the larger circle is $\tilde\delta(\eta)$ in the language of \cite{H}, along which the arc gets a positive Dehn twist.\label{mscheck}}\end{figure} triple $\tilde\delta(\eta),\gamma_1,\gamma_n$ as in Figure~\ref{mscheck}.

For multislides, $\gamma_1$ and $\gamma_n$ intersect at one transverse point. Consider the genus $g-1$ surface $\Sigma'$ obtained by replacing a neighborhood of $\gamma_1\cup\gamma_n$ with a disk marked with a point $p$. Behrens and Hayano show in sections 3.2 and 3.3 of \cite{BH} that an element of $\Kr\Phi_{\gamma_1}\cap\Kr\Phi_{\gamma_n}$ is the lift to $\mcg(\Sigma)$ (according to the re-attachment of the punctured torus at $p$) of what they call a \emph{line-pushing map} along an embedded circle in $\Sigma'$ based at $p$. By their lemma 3.16, a line-pushing map is a point-pushing map composed with a power of what they call a $\Delta$-twist, which is a positive half-twist along the boundary of the disk. As for handleslides, it is not hard to see that the effect of such a lift on any simple closed curve is a collection of slides over $\gamma_1$ and $\gamma_n$, or to check cases for the Dehn twists given by Equation (12) of that paper and for $\Delta$-twists. For this reason, the lemma is proved for multislides.

The last case is the shift move. To make it easier to read, the rest of this proof uses the indexing conventions of \cite{BH}, in which the initial fold merge occurs between elements $\gamma_k,\gamma_l\in\Gamma$ for $1<k<l$, and $\Gamma=(\gamma_1,\ldots,\gamma_l)$. Like multislide, shift is a move one can perform for any pair $\gamma_k,\gamma_l$ that intersects at one transverse point, sending \[(\gamma_1,\ldots,\gamma_k,\ldots,\gamma_l)\mapsto\left((\gamma_1,\ldots,\gamma_k,\gamma_l,\chi_0(\gamma_{k+1}),\ldots,\chi_0(\gamma_{l-1})\right).\] 
In \cite[Proposition~4.8]{BH}, there appears the formula \begin{equation}\label{eq:shiftformula}\chi_0=\varphi_0t_{\gamma_k}t_{\gamma_l}t_{\gamma_k}\Delta^m_{\gamma_k,\gamma_l}.\end{equation}
As constructed directly after the proof of \cite[Proposition~4.5]{BH}, the map $\varphi_0$ is the lift of a push map in the lower-genus surface obtained by replacing $\gamma_l$ with two marked disks, which again results in some collection of slides being applied to $(\gamma_{k+1},\ldots,\gamma_{l-1})$. The condition of being realized by slides is preserved by applying the other three Dehn twists, as shown in Figure~\hyperref[shift1]{4}, which at the top of each subfigure depicts the three twists in succession being applied to the circle $\varphi_0(\gamma_i)$ in a neighborhood of $\gamma_k\cup\gamma_l$, $k+1\leq l-1$ (the properly embedded arc $\varphi_0(\gamma_i)$ gets a Dehn twist along each dotted circle). At the bottom of the each subfigure, a sequence of slides achieves the same modification. Finally, the $\Delta$-twist $\Delta_{\gamma_k,\gamma_l}$ (which in Formula~\ref{eq:shiftformula} is raised to the power $m$) gets a similar treatment in Figure~\ref{deltatwist}.
\begin{figure}\capstart%
	\centering
	\subfloat[\ ]{\label{shift1}\includegraphics[width=0.4\linewidth]{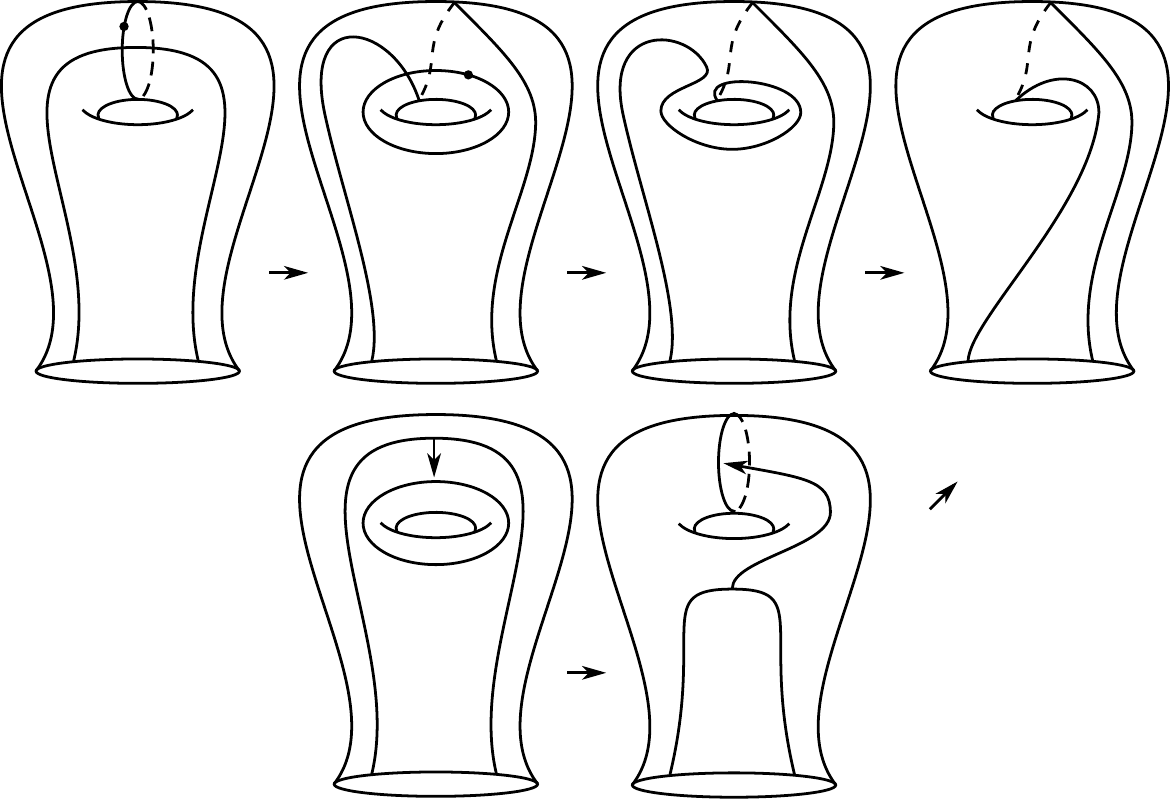}} \qquad \capstart
	\subfloat[\ ]{\label{shift2}\includegraphics[width=0.4\linewidth]{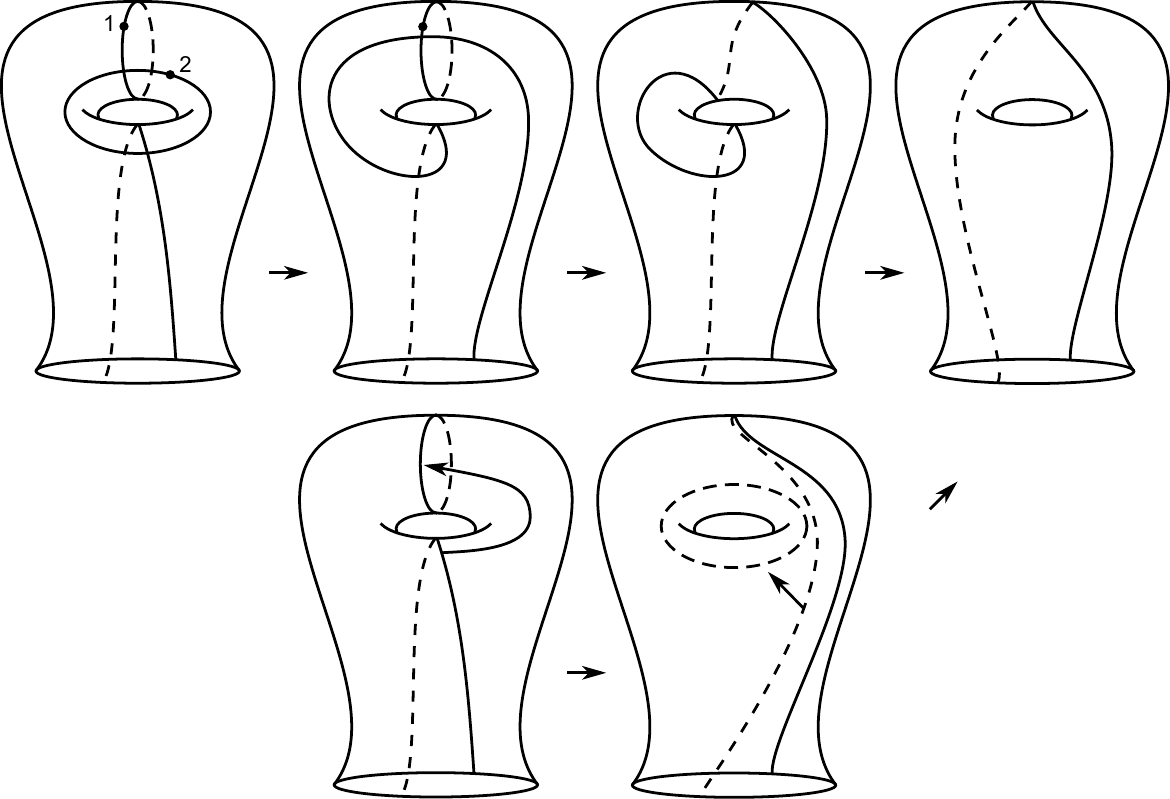}}
	\caption{\label{shift}Realizing part of the shift move by slides over $\gamma_k$ and $\gamma_l$.}
\end{figure}
\begin{figure}\capstart%
	\centering
	\subfloat[\ ]{\label{deltatwist1}\includegraphics[width=0.3\linewidth]{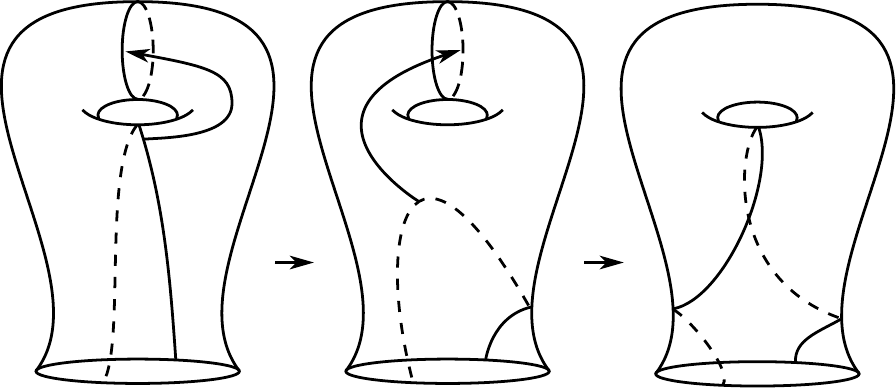}}\qquad \qquad \qquad \capstart
	\subfloat[\ ]{\label{deltatwist2}\includegraphics[width=0.3\linewidth]{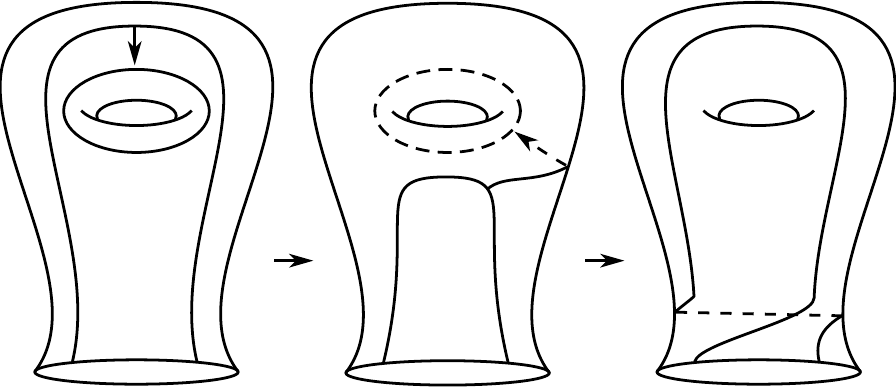}}
	\caption{\label{deltatwist}Checking that a $\Delta$-twist (as applied to an embedded circle) can be realized as a sequence of slides over $\gamma_k$ and $\gamma_l$.}
\end{figure}
\end{proof}
The main theorem of \cite{W2} states that the collection of surface diagrams coming from a fixed homotopy class of maps are all related by stabilization and the following \emph{genus-preserving moves}: diffeomorphism of $\Sigma$, isotopy of individual elements of $\Gamma$, handleslide, multislide, and shift. The following lemma refines this theorem, and should be a surprise to 3-manifold tolopogists, who may\begin{figure}\capstart%
	\centering
	\subfloat[Cerf diagram for 3-manifolds. ]{\makebox[6cm][c]{\includegraphics{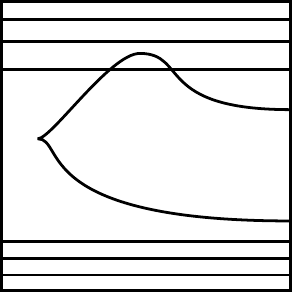}\label{cerf}}} \qquad \capstart
	\subfloat[Decorated critical surface for surface diagrams.]{\makebox[6cm][c]{\includegraphics{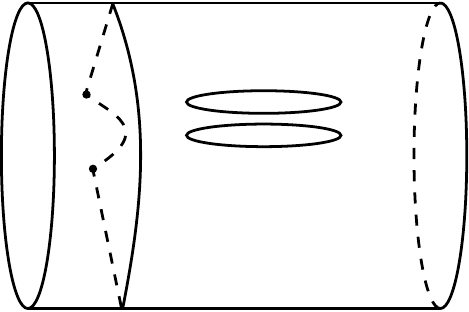}\label{cbd}}}
	\caption{\label{cerfcbd}Stabilization followed by handleslide. In the left, the lines are fold arcs traced out by Morse critical points. In the right, the arcs are places where the deformation is not injective on its critical locus. In both, time increases to the right.}
\end{figure} know that Heegaard splittings of closed 3-manifolds typically require stabilization before they become equivalent, even when they have the same genus; see for example \cite{HTT}. In that paper, the two Heegaard diagrams coming from inequivalent splittings are related by switching the roles of the $\alpha$ and $\beta$ circles, a hurdle that does not exist for surface diagrams. On a more fundamental level, consider a Heegaard stabilization that introduces an $\alpha$ circle $\alpha_n$, followed by a handleslide of $\alpha_n$ over another $\alpha$ circle. The Cerf graphic of the corresponding deformation $d$ of Morse functions contains a cusp, one of whose two adjacent folds intersects a third fold twice to form a bigon that cannot be eliminated by an $R_2$ deformation of $d$ that cancels the intersections (such a situation is described in the paragraph immediately following \cite[Proposition~A.4]{W2}; see also Figure~\ref{cerf}). For this reason, there is not a straightforward way to modify a Cerf graphic to switch the order of a stabilization followed by a handleslide. This is a kind of linking behavior that is unavoidable for 3-manifolds, but can be circumvented for surface diagrams.

See \cite[Theorem~6.5]{H} for the definitive treatment of stabilization, or \cite{W2} for a summary of its effects on a surface diagram and a detailed explanation of Figure~\ref{cbd}.
\begin{lemma}\label{stablem}Suppose two surface diagrams $(\Sigma,\Gamma)$ and $(\Sigma,\Gamma')$ are related by one of the following sequences of moves.\begin{enumerate}\item Stabilize, then perform genus-preserving moves, then de-stabilize.\item De-stabilize, then perform genus-preserving moves, then stabilize.\end{enumerate}Then there is a sequence of genus-preserving moves relating $(\Sigma,\Gamma)$ and $(\Sigma,\Gamma')$.\end{lemma}
\begin{figure}[h]\capstart\begin{center}\includegraphics{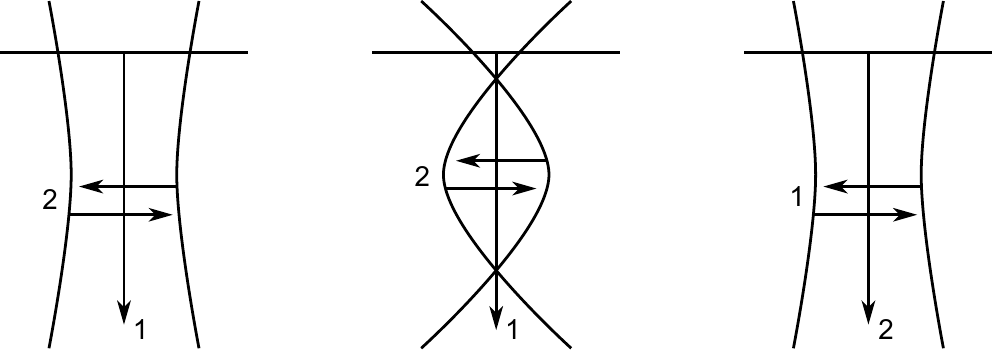}\end{center}\caption{Switching the order of a stabilization and a handleslide.\label{switchdef}}\end{figure}
\begin{proof}Case 1. The following arguments freely use language from \cite{BH,W2}. We first explain why it is possible to switch the order of stabilizations and handleslides without changing the fibration maps at either end of the deformation; switching stabilizations with shifts and multislides will be similar. It is important to note that though the order of moves switches, so that a handleslide move occurs in a lower-genus surface diagram which later stabilizes, there may not be an obvious way to see that handleslide for a given initial deformation. In \cite[Figure~20]{W2}, there appears a way to change the fold arc on which a flip occurs by a homotopy that fixes the endpoints of a given deformation. For this reason, one may assume the flipping moves for the stabilization occur on a fold arc other than the fold arcs that undergo merge or $R_2$ deformations for the subsequent genus-preserving move. The next step is to move the stabilization forward in the deformation past the genus-preserving move. This corresponds in Figure~\ref{cbd} to moving the immersion pair (the two circles) to be to the left of the immersion single (the twice-dotted circle). Figure~\ref{switchdef} has base diagrams with arrows that signal impending movements of fold arcs. It is meant to explain what happens when the circles in Figure~\ref{cbd} pass each other. When performing a stabilization, there occur two flipping moves, then a fold arc moves across the entire critical image in an $R_2$ deformation: the left side has this (horizontal) arc move down before the $R_2$ deformation of a handleslide. We want it to move down \emph{after} that $R_2$ deformation, as in the right side. If this is possible, then the proposed switch is possible, because the handleslide deformation is supported in the preimage of Figure~\ref{switchdef}. The middle figure represents an intermediate phase of the switch, in which the fold arc moves down in the middle of the handleslide $R_2$. The question is whether the proposed 2-parameter family of maps $M\to S^2$ exists, and its existence follows from an argument like the proof of \cite[Proposition~A.7]{W2}, where the disjointness condition \cite[Definition~A.1]{W2} is achieved at every stage of the modification by appropriately pushing away any intersections of (a one-parameter family of) vanishing sets over each point between the three sets of roughly vertical fold arcs in Figure~\ref{switchdef}. The arguments for switching a stabilization and a multislide or shift deformation are similar, but less involved because they involve a one-parameter family of deformations moving cusps into the higher genus side of the horizontal fold arc, instead of having a one-parameter family of $R_3$ deformations.

With this understood, move the stabilization forward until it occurs just before the de-stabilization and move the stabilization to occur on the same fold arc as the de-stabilization. This results in a deformation consisting of genus-preserving moves followed by a stabilization and de-stabilization at the same fold arc. Looking more closely at how this appears in base diagrams, this deformation consists of two flips, then an $R_2$ deformation, then a reverse $R_2$ deformation, then two inverse flips. During the two $R_2$ deformations, a system of reference paths from a reference point (whose fiber is $\Sigma$) to the higher-genus sides of the fold arcs remains disjoint from the fold image, so that there is an identification of vanishing cycles before and after canceling the intersections by $R_2$ moves.\footnote{This pair of $R_2$ deformations have fold arcs approaching each other from the side \emph{opposite} that expected of a handleslide deformation.} For this reason, the interval between the $R_2$ moves during which the critical image is embedded in the sphere can be contracted, so that the critical image remains immersed for the entire time between the flips. Now there appears the formation of two loops via flipping moves, which expand to contain all the cusps, and then shrink back down via inverse flipping moves. By a further homotopy of the deformation, decrease this expansion so there are no cusp-fold crossings, so that there occur two flips immediately followed by two inverse flips. These four flips can then clearly be eliminated from deformation, and this concludes Case (1).

Case 2. The argument is to first move the de-stabilization forward to occur just before the stabilization. This is even easier than the previous case, because it is a genus-increasing modification of the deformation in the sense that the genus-preserving moves are all transferred into higher-genus diagrams, but we can also rely on arguments from \cite{W2}: the swallowtails of the initial destabilization can be pushed forward toward the swallowtails of the final stabilization using Proposition~3.1 and Figure~20 from \cite{W2}, and Figure~18 depicts their cancellation. Lemma 3.8 allows the immersion locus that remains of the initial de-stabilization to be moved forward past all intervening genus-preserving moves. The result is what is called a \emph{genus-decreasing immersion pair} at the end of the deformation, which by \cite[Lemma~2.1]{W2} corresponds to a sequence of handleslides.\end{proof}
The section closes with an observation related to the topological partitions inhabited by generators.
\begin{prop}\label{hscor}The genus-preserving moves preserve the partition of $\cP$ into sectors, and the partition of sectors according to $H_1(M)$.\end{prop}
\begin{proof}The result is immediate for the cases of handleslide and multislide because of Lemma~\ref{slidelem}. However, the shift move reorders the elements of $\Gamma$, so it remains to show this reordering preserves the sector of a generator. To make sense of this statement, let $(\Sigma,\Gamma)$ and $X$ be the surface diagram and submanifold of $M$ corresponding to the map $f\co M\to S^2$ and let $(\Sigma,\Gamma')$ and $X'$ be the resulting surface diagram and submanifold corresponding to $f'$, where $f'$ comes from applying a shift homotopy to $f$. Then $X$ and $X'$ are isotopic submanifolds of $M$, diffeomorphic to $M\setminus\nu F_s$, where $\nu F_s$ is a neighborhood of a lower-genus fiber disjoint from the support of the homotopy $f\to f'$. Similarly, the reference fiber $\Sigma$ exists for both diagrams as elements of the same isotopy class of surfaces in $M$ (see \cite{W2} or \cite{BH} for details on the shift homotopy). Construct $\ba,\bb$ as in Section~\ref{basicdefs} as a copy of $\Gamma$ and a perturbation of $\Gamma$. According to Lemma~\ref{slidelem}, the shift move affects $\Gamma$ by a sequence of slides followed by a reordering, so for a reference generator and generator $\x_0,\x\in\csf(\Sigma,\ba,\bb)$, there is an obvious copy of $\x,\x_0$ after performing the slides, and the cycle $c_{\x\x_0}$ transforms by slides, preserving its homology class in $X$, and thus by definition the sector of $\x$ (recall Section~\ref{sectors} for the construction of $c_{\x\x_0}$). Let $\alpha_i',\beta_i',\gamma_i'$ be obtained from $\alpha_i,\beta_i,\gamma_i$ by the sequence of slides required by the shift move, for which $\gamma_l$ is assumed to be isotopic to a circle that intersects $\gamma_k$ transversely at one point (note that in this paper, $k=|\Gamma|$, which is denoted $l$ in \cite[Section~4.1.2]{BH}). Then the reordering, given by \[(\gamma_1,\ldots,\gamma_l,\gamma_{l+1}',\ldots,\gamma_k')\mapsto(\gamma_1,\ldots,\gamma_l,\gamma_k',\gamma_{l+1}',\ldots,\gamma_{k-1}'),\] corresponds to a change of reference generator from $\x_0$ to $\x_0'$: The entries of $\x_0$ lying in $\alpha_k'$ and $\beta_k'$, which by Corollary~\ref{topologicalpartitioncor} can be assumed to be points $p_{k,1}\in\alpha_k'\cap\beta_1'$ and $p_{k-1,k}\in\alpha_{k-1}'\cap\beta_k'$, are replaced with points $p_{k,l+1}\in\alpha_k'\cap\beta_{l+1}'$ and $p_{l,k}\in\alpha_l'\cap\beta_k'$. In particular, these four intersections are nonempty. In the construction of $c_{\x\x_0}$, one may choose $a_k$ to cross $p_{k,l+1}$ and choose $a_{k-1}$ to cross $p_{l,k}$. Now there is a choice of $c_{\x\x_0'}$ such that $c_{\x\x_0}$ and $c_{\x\x_0'}$ have a common subdivision.\end{proof}
\end{subsection}\end{section}

\begin{section}{Structure of the moduli spaces}\label{modulispaces}In this section there appear several results that are necessary to ensure the definition of $\fm$ makes sense. Much of the material here is simply commentary on how to adapt arguments of \cite[Sections~3-7]{L} to the current situation. For example, by the same arguments from \cite[Section~6]{L}, the moduli spaces $\mathcal{M}^A$ can be given coherent orientation systems using complete sets of paths, and the isomorphism class of $(\cP,\fm)(M,g,\sigma)$ is independent of this choice and the orientations of $\ba,\bb$.
\begin{subsection}{Index}
The substance of this section is Proposition~\ref{indexprop}, which follows from a close reading of \cite[Section~4]{L}.
\begin{df}For a positive element $A\in\pi_2(\arx,\y)$ with holomorphic representative $u\co(S,j)\to(W_{n+1},J)$, let $\ind(A)$ denote the index of the $D\overline{\partial}$ operator at $u$.\end{df}
\begin{df}Let $\bar{S}$ be a compact, smooth surface of genus $g$ with $b$ boundary components, and let $S$ be the result of removing $m$ points from the boundary and $n$ points from the interior. then $\chi(S)=2-2g-b-n-m/2$.\end{df}
\begin{prop}\label{indexprop}For any positive element $u\co S\to W_{n+1}$ in $\pi_2(\arx,\y)$ satisfying (\hyperlink{m0}{{\bf M0}})--(\hyperlink{m6}{{\bf M6}}) and $n\geq1$, let $\bar{S}$ be the closure of $S$ and let $e(D)$ denote the Euler measure of the domain of $u$. Then 
\begin{equation}\label{eq:indexn}\ind(A)=\frac{3-n}{2}k-\chi(\bar{S})+2e(D).\end{equation}\end{prop}
\begin{proof}
Here follows an adaptation of the proof of \cite[Equation 6]{L}, suitably generalized to allow $n>1$. Let $a_i^{\ell/(n+1)}$ be the boundary arc of $S$ that maps to $C_i^{\ell/(n+1)}$ for $1\leq\ell\leq n$, and let $b_i$ be the boundary arc that maps to $C_i^{(n+1)/(n+1)}$ under the map $u\co S\to W_{n+1}$. Going around the boundary circles of $S$, the boundary conditions on $u$ and the labeling of $\Delta_{n+1}$ imply that, for all $1\leq i\leq k$ and $1\leq\ell\leq n-1$, any boundary arc $a_i^{\ell/(n+1)}$ is followed by $a_j^{(\ell+1)/(n+1)}$ for some $j$ that depends on $u$. Similarly, any $a_i^{n/(n+1)}$ is followed by some $b_j$, and any $b_i$ is followed by some $a_j^{1/(n+1)}$. For this reason, the boundary arcs $a_i^{\ell/(n+1)}$ appear in sequence with $\ell$ increasing from 1 to $n$, followed by a $b_i$, so that each sequence (call it an $a_i$-sequence, of which there are $k$) can be interpreted like a single $a_i$ arc from the construction of $4\ltimes S$ in \cite{L}, and each individual $b_i$ can be interpreted like a $b_i$ from the same construction. The Euler characteristic calculation in our case also gives $\chi(4\ltimes S)=4\chi(\bar{S})-4k$: the first step is to glue two copies of $\bar{S}$ along the $k$ $a_i$-sequences ($\chi=2\chi(\bar{S})-k$), then double the result along its $2k$ boundary circles (each formed by a pair of $b_i$ arcs, doubling the Euler characteristic), then puncture at the $2k$ endpoints of the $b_i$ arcs in the resulting closed surface. Note this construction lacks punctures at the interior points of where the entries of each $a_i$-sequence meet. In our case there is also a quadruple operator $4\ltimes D\overline{\partial}$ at this curve whose index at the curve $4\ltimes\bar{S}$ is 
\begin{equation}\label{eq:index1}-\chi(4\ltimes\bar{S})+2c_1(A),\end{equation}
where $c_1(\bar{A})$ is defined as the pairing of the first Chern class of $\overline{u^*TW_{n+1}}$ with the fundamental class of $4\ltimes\bar{S}$ (this analogy works, even though the cylinders $C_i^\ell,C_j^\ell$ are not necessarily disjoint, because the calculations here and in \cite[Section 5]{Bo} are done in the pullback bundle and $u$ is an embedding).

In the same vein, $c_1(A)$ can be interpreted as a sum of Maslov type indices along the boundary arcs of $\bar{S}$,
\[c_1(A)=2\left(\sum\limits_{i=1}^{k}\sum\limits_{\ell=1}^{n}\mu\left(a_i^{\ell/(n+1)}\right)-\sum\limits_{i=1}^{k}\mu(b_i)\right)-(n-1)k,\]
Where the last $-(n-1)k$ term comes from the punctures at the corners of $\Delta_{n+1}$ corresponding to those that were neglected before, and the sum of Maslov indices can be recast as twice the Euler measure of the domain of $u$. Thus,
\[\begin{array}{rcl}
\ind(D\overline{\partial})(A)&=&\frac{1}{4}\left(-4\chi(\bar{S})+4k+2c_1(A)\right)\vspace{0.5em}\\&&
k-\chi(\bar{S})+\frac{1}{4}\cdot2\left(2\left(\sum\limits_{i=1}^{k}\sum\limits_{\ell=1}^{n}\mu\left(a_i^{\ell/(n+1)}\right)-\sum\limits_{i=1}^{k}\mu(b_i)\right)-(n-1)k\right)\vspace{0.5em}\\&&
k-\chi(\bar{S})+2e(D)-\frac{n-1}{2}k,\end{array}\]
which is the required formula.\end{proof}
Note this index formula agrees with \cite[Equation~6]{L} for $k=g$, $n=1$ and Lipshitz's index formula for triangles when $n=2$, and it also agrees with what would result from applying the Riemann-Hurwitz formula to the index formula that appears in \cite{S} for the $(n+1)$-gon in $\Sym^k(\Sigma)$ corresponding to $u$, using the fact that the algebraic intersection number $\iota$ with the \emph{fat diagonal} is precisely the number of order 2 branch points of $\pi_{\Delta_{n+1}}(u)$ (suitably perturbed). For details on this so-called \emph{tautological correspondence} between curves in $W_{n+1}$ and polygons in $\Sym^k(\Sigma)$, see \cite[Sections 4.3 or 13]{L} or the discussion preceding \cite[Theorem~5.2]{S}. However, the immersed tori in $\Sym^k(\Sigma)$ corresponding to the $k$-tuples $\Gamma^{\ell/(n+1)}$ have isolated points at which they are not totally real, so it is not clear that $D\overline{\partial}$ is Fredholm for generic perturbations of $\Sym^k(j_\Sigma)$. Note also that Sarkar's formula, as applied to a curve $u\co S\to W_{n+1}$, only depends on the homology class of $u$. The only part of Formula~\ref{eq:indexn} that does not clearly depend only on the homology class of $u$ is $\chi(\bar{S})$; this brings up two natural questions.
\begin{quest}\label{indq}For an embedding $u\co S\to W_{n+1}$, is $\chi(\bar{S})$ determined by the homology class of $u$?\end{quest}
\begin{quest}Is there a symmetric reformulation of the invariants defined in this paper? In other words, does the tautological correspondence lead to an equivalent theory of polygons in $\Sym^k(\Sigma)$?\end{quest}
One way to address Question~\ref{indq} is to determine whether there is an adaptation of the constructions that appear in \cite[Section~4.2]{L} (and also in the proof of \cite[Lemma~4.1$'$]{L2}) or to adapt \cite[Proposition~10.9]{L} to our maps. Perhaps one weak piece of evidence for an affirmative answer is the effect on the index of adding $[\Sigma]$ to the homology class of $u$:
\begin{prop}\label{addsig}For a map $u\co S\to W_{n+1}$ satisfying (\hyperlink{m0}{{\bf M0}})--(\hyperlink{m6}{{\bf M6}}) and $n\geq1$, choose $m$ generic points $p1,\ldots,p_m$ in the interior of $\Delta_{n+1}$ and resolve the $mk$ intersections between $u(S)$ and $\Sigma\times\{p_1,\ldots,p_m\}$, calling the resulting map $u'$. Then \begin{equation}\label{eq:addsigma}\ind(u')=\ind(u)+2m(1-g+k),\end{equation} where $g$ is the genus of $\Sigma$.\end{prop}
\begin{proof}According to Proposition~\ref{indexprop},
\[\ind(u')=\frac{3-n}{2}k-(\chi(\bar{S})+m(2-2g)-2mk)+2e(D+m[\Sigma]).\]
Then $e(m[\Sigma])=m(2-2g)$, and additivity of the Euler measure yields the result.\end{proof}
Because of the additive nature of Formula~\ref{eq:addsigma}, the index can be extended to non-positive domains.
\begin{df}\label{inddef}For $u\in\pi_2(\arx,\y)$ satisfying (\hyperlink{m0}{{\bf M0}})--(\hyperlink{m6}{{\bf M6}}) and $n\geq1$, let $n$ be the least coefficient of the domain of $u$, and let $m=\max\{0,-n\}$. Let $u'$ be the map obtained by adding $m$ copies of $\Sigma$ as in Proposition~\ref{addsig}, and define 
\[\ind(u)=\ind(u')-2n(1-g+k).\]
\end{df}
\begin{df}Let $\cM_i(\arx,\y)$ denote the moduli space of index $i$ elements of $\cM(\arx,\y)$. In this notation, $\widehat{\cM}_i(\x,\y)=\cM_{i+1}(\x,\y)/\R$.\end{df}\end{subsection}
\begin{subsection}{Compactness}
The existence of generic admissible $J$ that achieves transversality for a given $A\in\pi_2(\arx,\y)$ satisfying (\hyperlink{m0}{{\bf M0}})--(\hyperlink{m6}{{\bf M6}}) follows essentially word-for-word from the arguments of \cite[Section~3]{L}, as do the necessary gluing results from \cite[Appendix A]{L}. The terms \emph{concatenation} and \emph{height two holomorphic building} below are defined in a way that makes them completely analogous to the terms as they appear in \cite{L}.
\begin{df}\label{concatdef}For integers $n_1\geq0$ and $n_2>0$, consider a pair \[(u,v)\in\pi_2(\arx_1,\y_1)\times\pi_2(\arx_2,\y_2),\] where $\arx_1=\left(\x_1^{1/(n_1+1)},\ldots,\x_1^{n_1/(n_1+1)}\right)$, $\arx_2=\left(\x_2^{1/(n_2+1)},\ldots,\x_1^{n_2/(n_2+1)}\right)$ and the representative of $\y_1$ coincides with that of $\x=\x_2^{i/(n_2+1)}$ for some $i$ (if $n_1=0$ then omit $\arx_1$). We assume the main perturbations chosen for $u$ and $v$ are such that $\Gamma^\frac{1}{n_1+1}=\Gamma^\frac{i}{n_2+1}$ and $\Gamma^\frac{n_1+1}{n_1+1}=\Gamma^\frac{i+1}{n_2+1}$, and such that the sequence of circles 
\[\left(\Gamma^\frac{1}{n_2+1},\ldots,\Gamma^\frac{i-1}{n_2+1},\Gamma^\frac{1}{n_1+1},\ldots,\Gamma^\frac{n_1+1}{n_1+1},\Gamma^\frac{i+1}{n_2+1},\ldots,\Gamma^\frac{n_2+1}{n_2+1}\right)\]
comes from a main perturbation. Then the \emph{concatenation} $u\ast_\x v$ of $u$ and $v$ is a representative of the element $[u]+[v]\in\pi_2(\arz,\y_2)$ given by gluing representatives of $u,v$ at their common ends $\y_1,\x_2^{i/(n_2+1)}$, so that 
\[\arz=\left(\x_2^{1/(n_2+1)},\ldots,\x_2^{(i-1)/(n_2+1)},\arx_1,\x_2^{(i+1)/(n_2+1)},\ldots,\x_2^{n_2/(n_2+1)},\y_2\right).\]\end{df}
\begin{df}For integers $n_1\geq0$ and $n_2>0$ as in Definition~\ref{concatdef}, a \emph{two-story holomorphic building in the homotopy class $A_1+A_2$} is a pair \[(u,v)\in\cM_0^{A_1}(\arx_1,\y_1)\times\cM_0^{A_2}(\arx_2,\y_2)\] suitable for concatenation. If $n_i=1$, use the moduli space $\widehat{\cM}_0^{A_i}(\x_i,\y_i)$.\end{df}
\begin{thm}\label{compactness}\label{cptthm}Assume $J$ is a generic, admissible almost complex structure on $W_{n+1}$ that achieves transversality for a given $A\in\pi_2(\arx,\y)$ satisfying (\hyperlink{m0}{{\bf M0}})--(\hyperlink{m6}{{\bf M6}}). Then:\begin{enumerate}
\item The moduli spaces $\cM^A(\arx,\y)$ are smooth manifolds of dimension $\ind(A)$, for $\ind(A)\leq2$. The moduli space $\cM(\y)$ consists of finitely many points. 
\item The moduli spaces $\cM_1^A(\arx,\y)$ and $\widehat{\cM}_1^A(\x,\y)$ have compactifications whose boundary points are height two holomorphic buildings in the homotopy class $A$.\end{enumerate}\end{thm}
\begin{proof}Assertion (1) is a basic consequence of transversality. As for (2), it is conceivable for all the possible types of degenerations appearing in \cite{L} to occur (and slightly more): using the notation of Definition~\ref{concatdef}, there is an obvious version of the \emph{level splitting} that appears in \cite[Section~7]{L} for $n_1,n_2>0$, in which a $k$-tuple of Reeb chords forms between distinct sides of $\Delta_{n_1+n_2}$ (more than two sides is ruled out by the standard index argument), and also there are the possibilities of Deligne-Mumford type degenerations of $S$ (ruled out for index reasons, because such degenerations have codimension 2 in the space of all holomorphic curves), disk or sphere bubbling, and bubbling off copies of $\Sigma$. In $W_{n+1}$, disk bubbling is meant to describe the formation of a holomorphic disk in $W_{n+1}$ whose boundary lies in a single cylinder. This is impossible, since each cylinder comes from an embedded homologically essential circle in $\Sigma$. Sphere bubbling is ruled out because $\pi_2(W_{n+1})$ is trivial. Thus, verification of conditions other than (\hyperlink{m2}{{\bf M2}}) and (\hyperlink{m3}{{\bf M3}}) for limiting curves is just like in \cite[Proposition 7.1]{L}. We now describe boundary points that correspond to degenerations that do not satisfy (\hyperlink{m2}{{\bf M2}}) and (\hyperlink{m3}{{\bf M3}}).

\emph{Annoying curves} are components of $S$ mapped by $u$ to a fiber of $\pi_{\Delta_{n+1}}$. Because of our choice of (\hyperlink{j5}{{\bf J5}}), annoying curves on the interior of $W_{n+1}$ are not generic, so that if a sequence limits to a curve not satisfying (\hyperlink{m2}{{\bf M2}}), any offending component lies in $\partial W_{n+1}$, and has boundary in the cylinders there, leading to a violation of condition (\hyperlink{m3}{{\bf M3}}), where boundary arcs travel first along cylinders, then along the boundary of the annoying curve, then again along the same cylinders. The fact that the obvious inclusion map 
\[\Z\langle\gamma_1\rangle\oplus\cdots\oplus\Z\langle\gamma_k\rangle\to H_1(\Sigma)\]
can have nontrivial kernel means such curves are unavoidable. Just like in Heegaard-Floer theory, level splitting occurs when slits form in the domain of $u$ along elements of $\Gamma^\frac{\ell}{n+1}$ and $\Gamma^\frac{\ell'}{n+1}$ for some $\ell,\ell'$ and limit to elements of the same to form a concatenation of two domains. When $\ell=\ell'$, inserting an appropriate neck along $\Sigma\times c$ (where $c$ is a properly embedded arc in $\Delta_{n+1}$ with boundary in $e^{\ell/(n+1)}$) yields a bubbling off of a 1-gon containing an annoying curve from that side of $W_{n+1}$. This is the formation of a height two holomorphic building with $n_1=0$, except the cylinders where the two levels meet must be slightly perturbed to meet the definitions of $W_1$ and $W_{n+2}$; such a perturbation corresponds to an arbitrarily small perturbation of the almost complex structures of the two pieces near the relevant cylinders, which by transversality corresponds to a trivial cobordism of moduli spaces. 

With the assumption $\ell\neq\ell'$, the conditions (\hyperlink{m0}{{\bf M0}})--(\hyperlink{m6}{{\bf M6}}) are verified along the lines of \cite[Proposition 7.1]{L}, with no perturbation of cylinders necessary.
\end{proof}
\begin{ex}Here is an example of a degeneration due to Lipshitz which, along with a subsequent suggestion from Tim Perutz, was instrumental in the author's formulation of Theorem~\ref{thm}. It is important to note that neither Lipshitz nor Perutz have commented to the author about the validity of this paper's arguments or main result.

\begin{figure}[b]\capstart\begin{center}\includegraphics{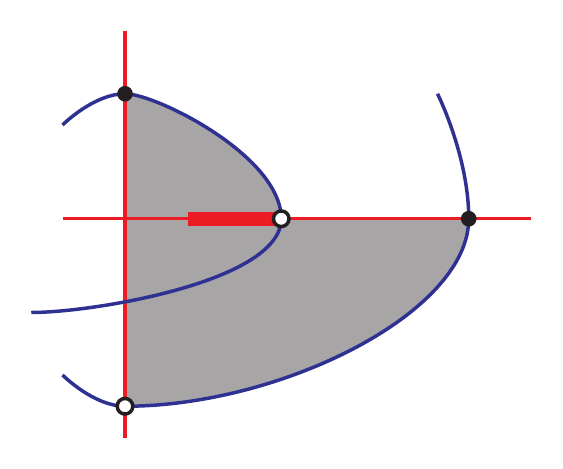}\end{center}\caption{A horizontal slit forms along $\alpha_1$ in the domain of $\widehat{\cM}_1(\x,\y)$, with one end approaching $\alpha_2$. \label{lex}}\end{figure} Figure~\ref{lex} describes a one-dimensional moduli space of holomorphic curves in $W_2$. At one end of this moduli space, there is a slit in the horizontal red arc, which we call $\alpha_1$, that approaches the intersection $\alpha_1\cap\alpha_2$. At the other end, the slit travels down along the blue arc toward $\alpha_2$. To those who know Heegaard Floer theory, the latter degeneration is a familiar instance of trajectory breaking; in our language it corresponds to level splitting according to an arc connecting the sides $e_1,e_2\subset \Delta_2$. The former, pictured, degeneration corresponds to level splitting according to an arc from $e_1$ to itself. The annoying curve in this instance is a constant pair of 1-gons whose domains are both $\alpha_1\cap\alpha_2$ before perturbation.\end{ex}
\end{subsection}
\begin{subsection}{Admissibility}
For arbitrary $(\arx,\y)$, in order to have a weak $A_\infty$ algebra with $\Z_2$ coefficients, it is necessary to guarantee that there are only finitely many classes in $\pi_2(\arx,\y)$ that contain holomorphic representatives. The property of being an \emph{admissible} surface diagram will suffice. The discussion is somewhat simplified compared to that of \cite{L,OS} because of the lack of a basepoint and the fact that the $k$-tuples $\Gamma^t$ are isotopic for varying $t$.
\begin{df}The surface diagram $(\Sigma,\Gamma)$ is \emph{admissible} if every nontrivial linear combination of regions $A\subset\Sigma$ such that $\partial A=\sum\limits_{i=1}^ka_i[\gamma_i]$ for some integers $a_i$ has positive and negative coefficients. We also require that $(\Sigma,\Gamma)$ comes from a map $f\co M\to S^2$ that is homotopic to a constant map, and $\Sigma$ has genus $g\geq3$.\end{df}
The Heegaard-Floer version of the above definition might be weak admissibility, modulo base point, for all Spin$^c$ structures. The requirements on $f$ and $g$ were explained in Section~\ref{basicdefs}: We require $f$ to be nullhomotopic because every smooth, closed oriented 4-manifold has a surface diagram in that distinguished homotopy class (\cite[Corollary~1]{W1} and $H_2(M)$ is fully expressed in any nonempty set of trajectories $\pi_2(\x,\y)$ (Proposition~\ref{h2prop}).
\begin{lemma}Any surface diagram is isotopic to one that is admissible.\end{lemma}
\begin{proof}Let $\cD=\{D_1,\ldots,D_m\}$ denote the set of path components of $\Sigma\setminus\Gamma$. In the integer lattice generated by $\cD$, choose a basis $(q_1,\ldots,q_b)$ for the subspace $Q$ of elements whose boundary is a sum of elements of $\Gamma$, which is isomorphic to the kernel of an inclusion map
\[\Kr\left(\bigoplus\limits_{i=1}^k\Z\langle\gamma_i\rangle\to H_1(\Sigma)\right)\cong H_2(X),\]
so that a domain is uniquely specified by its boundary. The rest of the proof follows as in \cite[Lemma~5.4]{OS}, winding along elements of the standard basis for $H_1(\Sigma)$.\end{proof}
\begin{lemma}\label{finitesums}When $(\Sigma,\Gamma)$ is admissible, there are only finitely many positive elements of $\pi_2(\arx,\y)$.\end{lemma}
\begin{proof}For any pair of positive elements $A,B\in\pi_2(\arx,\y)$, the difference of their domains has boundary given by a sum of elements of the $k$-tuples $\Gamma^{\ell/(n+1)}$. Performing the obvious isotopy reversing the main perturbation yields an element $q\in Q$, which, if nontrivial, has positive and negative coefficients. Then the usual finiteness argument follows: In \cite[Lemma~4.12]{OS}, which does not make use of the index or first Chern class, there is a proof that there is a volume form on $\Sigma$ such that the signed area of every element of $Q$ is zero. This puts a lower bound on each of the coefficients of $A-B$. This and the condition that the signed area of $A-B$ is zero gives an upper bound. Now suppose $q$ is trivial, which means $A-B$ lies in the subset of $\Gamma$ swept out by the main perturbation (these so-called \emph{thin domains} appear repeatedly; see for example Definition~\ref{thin}). The definition of \emph{main perturbation}, specifically that each copy of $\gamma_i$ transversely intersects $\gamma_i$ at two points, guarantees the same kind of admissibility criterion, and resulting finiteness result, as for domains that yield nontrivial elements of $Q$.\end{proof}
\end{subsection}\end{section}

\begin{section}{Definition of the algebra}\label{algdef}
Fix an admissible surface diagram $(\Sigma,\Gamma)$ for $M$.
\begin{df}Recall the definition of $\cP$ from Section~\ref{spaces}. For an $n$-tuple of generators $\arx$ where $n>1$, define the map $\fm=\left(\fm_n\co\cP^{\otimes n}\to\cP\right)$,
\[\fm_n(\arx)=\sum\limits_\y\#\cM_0(\arx,\y)\cdot\y.\]
For $n=1$,
\[\fm_1(\x)=\sum\limits_\y\#\widehat{\cM}_0(\x,\y)\cdot\y.\]
For $n=0$, set $\fm_0(\x)=\#\cM(\x)$.\end{df}
\begin{lemma}The pair $(\cP,\fm)(M,g)$ is a weak $A_\infty$ algebra.\end{lemma}
\begin{proof}This is a direct consequence of Theorem~\ref{cptthm} and admissibility. The $A_\infty$ relations are precisely a $\Z_2$ count of $\partial\cM_1(\arx,\y)$ or $\partial\widehat{\cM}_1(\x,\y)$.\end{proof}
Note that the $H_1(M)$ partition of sectors discussed in Section~\ref{secth1m} takes the ordering of $\Gamma$ into account because $a$ comes from comparing a generator $\x$ with the reference generator $\x_0$, which itself is pinned down by the ordering of $\Gamma$. These moduli spaces are all orientable using a complete set of paths as in Heegaard Floer theory, so it seems there should be a choice of sign in Equation~\ref{eq:ainftyrlns} allowing the theory to have integer coefficients, though an accompanying grading on the entire set of generators will require further work.
\end{section}
\begin{section}{Isotopy}\label{isotopy}Here begins the proof that the algebra $(\cP,\fm)(M,g)$ is a diffeomorphism invariant of $M$.
Recall the notation $\ba,\bb$ from Definition~\ref{vocab}. Similar to \cite[Section~9]{L}, a basic isotopy of an element $c$ of $\ba,\bb$ or $\Gamma$ is, for some $i$, an isotopy $c_t\co S^1\times[0,1]\to\Sigma\times[0,1]$ with one of the following properties.\begin{itemize}
\item(Type 1)\hypertarget{t1}\ The isotopy can be realized by an ambient isotopy in $\Sigma$.
\item(Type 2)\hypertarget{t2}\ The isotopy introduces one pair of transverse intersections between $c_t$ and some circle other than itself by a finger move creating a pair of intersections between $c$ and some other circle.
\item(Type 3)\hypertarget{t3}\ The isotopy appears as in Figure~\ref{r3}.\end{itemize}
\begin{figure}\capstart\begin{center}\includegraphics{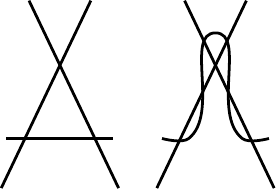}\end{center}\caption{A Type 3 isotopy of an element of $\Gamma$, $\ba$ or $\bb$.\label{r3}}\end{figure}
The Type 3 isotopy may appear strange; its particular form is chosen so that the second sentence of the proof of \cite[Lemma~9.2]{L} can be achieved.
\begin{df}An \emph{admissible isotopy} is a sequence of basic isotopies and their inverses such that the embellished diagram is admissible before and after each basic isotopy.\end{df}
\begin{lemma}\label{connect}Suppose $(\Sigma,\Gamma)$ and $(\Sigma,\Gamma')$ are isotopic admissible surface diagrams. Then there is an admissible isotopy sending $(\Sigma,\Gamma)$ to $(\Sigma,\Gamma')$.\end{lemma}
\begin{proof}The proof of \cite[Proposition~7.2]{OS} will suffice, performing appropriate isotopies of pairs $(\alpha_i,\beta_i)$ corresponding to $\Gamma_i$, neglecting all the trouble coming from the need to avoid a base point.\end{proof}
\begin{prop}\label{isotlem}Suppose $(\Sigma,\Gamma)$ and $(\Sigma,\Gamma')$ are two diagrams that differ by an admissible isotopy. Then the respective algebras $(\cP,\fm)(M,g)$ and $(\cP',\fm')(M,g)$ are homotopy equivalent.\end{prop}
By Lemma~\ref{connect}, the result follows from invariance under admissible isotopy of, say, $\gamma_1$, yielding the new surface diagram $(\Sigma,\Gamma')$. The following argument follows ideas from the proofs of \cite[9.1-9.5]{L}: Let $J,J'$ be almost complex structures on $W_{n+1}$ and $W'_{n+1}$ satisfying (\hyperlink{J1}{{\bf J1}})-(\hyperlink{J5}{{\bf J5}}) for $(\Sigma,\Gamma)$ and $(\Sigma,\Gamma')$, respectively. For $n>0$ and $p$ and incoming corner of $W_{n+1}'$, choose $T>0$, and let 
\[W^{n,n',\ell}=(\Sigma\times[0,1]\times\R,\omega),\]
decorated with a $k$-tuple of Lagrangian cylinders that agree with $\Gamma^\frac{n+1}{n+1}$ near $-\infty$ and ${\Gamma'}^\frac{\ell}{n'+1}$ near $+\infty$. The strip $W^{n,n',\ell}$ also comes equipped with an almost complex structure $J^{n,n',\ell}$ satisfying (\hyperlink{j1}{{\bf J$^\ell$1}}), (\hyperlink{j2}{{\bf J$^\ell$2}}) and (\hyperlink{j4}{{\bf J$^\ell$4}}) that agrees with $J^{(n+1)/(n+1)}$ for $\R$-coordinates less than $T$, and agrees with ${J'}^{\ell/(n'+1)}$ for $\R$-coordinates greater than $T$. We call this manifold an \emph{interpolation strip}.

For fixed $n>0$, one can read each summand in Equation~\ref{eq:ainftyrlns} as gluing instructions for making a manifold with cylindrical ends by making $\fm_i$ correspond to $W_i'$ and $\fm_j$ correspond to $W_j$, and gluing according to composition of functions, inserting interpolation strips wherever they are needed to make the cylinders meet smoothly. For example, the expression
\[\fm_3'\left(a_1,\fm_2(a_2,a_3),a_4\right)\]
corresponds to the manifold in Figure~\ref{isotex}, where the incoming ends are at the left and the outgoing end is at the right.
\begin{figure}\capstart\begin{center}\includegraphics{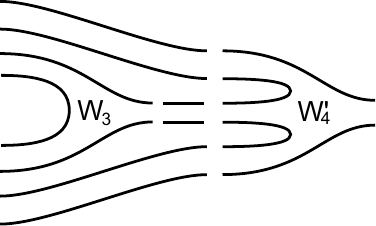}\end{center}\caption{Diagram of the target space used in the isotopy map specified by a summand in the $A_\infty$ relations. All three strips are interpolation regions specified by the corners of $W_4'$.\label{isotex}}\end{figure}
Enumerate the manifolds corresponding to the addends of Equation~\ref{eq:ainftyrlns} for fixed $n$ as $W_n^1,\ldots,W_n^{m(n)}$ and let $\cM_0^{n,p}(\arx,y)$ denote the moduli space of index 0 holomorphic curves $u\co S\to W_n^p$ satisfying (\hyperlink{m0}{{\bf M0}})-(\hyperlink{m6}{{\bf M6}}), using the unique representatives of the entries of $\arx$ in $\cP$ at the relevant corners of $W^p_n$. For an admissible isotopy that sends the diagram $(\Sigma,\Gamma^0)$ to $(\Sigma,\Gamma^1)$, define the isotopy map $\ff^{01}$ by the formula
\[\ff_n^{01}(\arx)=\sum\limits_{p=1}^{m(n)}\sum\limits_\y\#\cM_0^{n,p}(\arx,\y).\]
Finiteness of this sum follows just as in Lemma~\ref{finitesums}, where the points $z_i$ are chosen so they do not intersect any of the circles during the isotopy.
\begin{rmk}The above construction may seem strange when one considers gradings, because it seems it would reduce the grading (if it can be defined) by 2. For this reason, if one were to put a $\Z$ grading on the generators of $\cP$ instead of just a $\Z_2$ grading, it might be required to think of this as a morphism from $(\cP^0,\fm^0)$ to the shifted algebra $(\cP^0,\fm^0)[2]$, which may in the end mean that any index that can be put on all generators simultaneously would only be defined up to a global shift. There are similar observations for the maps in Section~\ref{slides}. The author intends to return to this in later work.\end{rmk}
\begin{lemma}\label{isotmorlem}The isotopy map $\ff^{01}$ is an $A_\infty$ morphism.\end{lemma}
\begin{proof}The moduli space $\cM^{n,p}_1(\arx,\y')$ of index 1 holomorphic curves in $W^p_n$ satisfying (\hyperlink{m0}{{\bf M0}})-(\hyperlink{m6}{{\bf M6}}) is a smooth 1-dimensional manifold that has the same bubbling properties as $W_{n+1}$, so it has a compactification whose boundary consists of height two holomorphic buildings, with level splitting according to arcs in $\Delta_{n+1}$. For this reason, the left hand side minus the right hand side of Equation~\ref{eq:morphismrlns} is a count of 
\[\partial\bigsqcup\limits_{p=0}^{m(n)}\bigsqcup\limits_\y\overline{\cM^{n,p}_1(\arx,\y)},\]
which is zero for each $n$.\end{proof}
\begin{lemma}\label{cylchoices}Different choices of cylinders and almost complex structures in the strips $W^{n,n'n\ell}$ yield homotopic isotopy maps.\end{lemma}
\begin{proof}This follows from an argument much like the proof of \cite[Lemma~9.4]{L}, except the analogous moduli space $\cup_t\cM_{0,t}(\x^1,\y^2)$ has eight more types of ends coming from terms corresponding to $\fm_0$: there could be level splitting  of $J_1,J_2,J,$ or $J'$-holomorphic curves on either side of the strip, but all such contributions cancel with contributions from opposite sides of the square of almost complex structures he defines, because of independence of $(\cP,\fm)$ under generic perturbations of almost complex structures.\end{proof}
\begin{lemma}\label{isotmorcomp}Suppose there is an admissible isotopy sending $(\Sigma,\Gamma^0)$ to $(\Sigma,\Gamma^1)$ to $(\Sigma,\Gamma^2)$. Then $\ff^{12}\circ\ff^{01}$ is homotopic to $\ff^{02}$.\end{lemma}
\begin{proof}By Lemma~\ref{cylchoices}, the manifolds obtained by concatenating the interpolation strips as defined for $\ff^{01}$ and $\ff^{12}$ can serve as the concatenation strips used to define $\ff^{02}$, since such choices do not affect its homotopy class. With this in mind, it is straightforward to pair up all of the manifolds $W^p_n$ (and thus the holomorphic curves counted therein) that appear in the composition according to Equation~\ref{eq:comp}, except for precisely those used in the definition of $\ff^{02}$.
\end{proof}
\begin{proof}[Proof of Proposition~\ref{isotlem}]According to the Lemmas~\ref{isotmorlem}, \ref{cylchoices}, and \ref{isotmorcomp}, the isotopy morphism $\ff^{01}$ for the given admissible isotopy and the isotopy morphism $\ff^{10}$ for its inverse compose to give morphisms $\ff^{01}\circ\ff^{10}$ and $\ff^{10}\circ\ff^{01}$ that are homotopic to the isotopy morphisms $\ff^{00}$ and $\ff^{11}$ coming from trivial isotopies. As in Lemma~\ref{isotmorcomp}, the manifolds used in the definition of the isotopy morphisms for a trivial isotopy come in identical pairs for $n>1$, so those maps are the zero map, while for $n=1$ it is not difficult to see the only curves that contribute are the so-called \emph{trivial disks}, copies of $\x\times\Delta_2$, so that those maps are identity maps.\end{proof}
\end{section}
\begin{section}{Handleslide, multislide and shift}\label{slides}
According to Lemma~\ref{slidelem}, the three moves in the title of this section are realized in a surface diagram by sequences of what look like three-dimensional handleslides, in a modification we will simply call a \emph{slide} of one circle over another. The shift move has one additional re-ordering of vanishing cycles that will also be addressed, thought the relevance of that discussion is mainly for later work. To slide $\gamma_i$ over $\gamma_j$, choose an embedded path from a point in $\gamma_i$ to a point in a homologically distinct $\gamma_j$, whose intersection with $\gamma_i$ and $\gamma_j$ is precisely its endpoints, and then replace $\gamma_i$ with the connect sum $\gamma_i\#\gamma_j$ specified by that path. Much of this section is concerned with curves in strips - using the notation $\ba,\bb$ from Definition~\ref{vocab}, let $W_{\alpha\beta}$ denote the cylindrical manifold $\Sigma\times[0,1]\times\R$ decorated with the cylinders $\Gamma^{t_1}\times\{0\}\times\R$, $\Gamma^{t_2}\times\{1\}\times\R$. Denote by $\bb^H$ the set of circles that comes from performing a slide on an element of $\bb$. There is a canonical inclusion of generators that respects the partition of generators into sectors in the sense that the sectors in the second group can still be interpreted as elements of $H_1(X)$ using the obvious ``same" reference generator as the one for $(\ba,\bb)$, and there is a natural correspondence of domains $\pi_2^{\alpha\beta}(\x,\y)\to\pi_2^{\alpha\beta^H}(\x,\y)$ (see, for example, \cite[Figure~41]{OSS}). The handleslide move modifies an element $\gamma\in\Gamma$ by sliding over one circle, or once over each of two disjoint circles; see for example \cite[Section~2.1]{W2} or \cite[Proposition~4.9]{BH}. In this move, the first slide may introduce a self-intersection in $\gamma$, and the second always removes it (see, for example, \cite[Figure~6]{W2}). The other two moves are realized by slides that preserve the condition that $\gamma$ is a simple closed curve. We address this by allowing $\gamma$ to undergo a pair of slides as needed, explaining why the maps behave as required: on the level of generators, it is essentially what one would expect from a composition of handleslide maps in Heegaard-Floer theory. Like in Section~\ref{isotopy}, there first appears a map $\cP\to\cP^H$, where $\cP^H$ is the $\Z_2$ vector space generated by the collection of generators coming from the surface diagram that has undergone one or a pair of slides. Then that map is inserted appropriately to produce a homotopy equivalence of $A_\infty$ algebras. Here follows the main result of the section.

\begin{prop}[c.f. Proposition~11.2 of \cite{L}]\label{lil11.2}Let $(\Sigma,\Gamma')$ be obtained from the admissible surface diagram $(\Sigma,\Gamma)$ by shift, handleslide or multislide. Then the algebra $(\cP,\fm)$ defined using $(\Sigma,\Gamma)$ is homotopy equivalent to the algebra $(\cP',\fm')$ defined using $(\Sigma,\Gamma')$.\end{prop}

The result follows from repeated application of the following
\begin{lemma}\label{slideinv}Let $(\Sigma,\Gamma^H)$ be obtained from the admissible surface diagram $(\Sigma,\Gamma)$ by sliding $\gamma_1$ over $\gamma_n$ or once each over the disjoint pair of circles $\gamma_n,\gamma_{n'}$. Then the algebra $(\cP,\fm)$ defined using $(\Sigma,\Gamma)$ is homotopy equivalent to the algebra $(\cP^H,\fm^H)$ defined using $(\Sigma,\Gamma^H)$.\end{lemma}

It seems possible to define maps that take care of the entire sequence of slides that comprises a multislide, handleslide or shift at once, but this paper takes the route suggested by Lemma~\ref{slideinv} for ease of exposition and to illustrate the robust nature of the construction of $(\cP,\fm)$. The idea is to use a restricted version of the Heegaard-Floer triangle maps for a single slide (or pair of slides) first applied to $\bb$, then the same applied to $\ba$; the result is a map sending generators to those in an embellished diagram that came from performing the same upon $(\Sigma,\Gamma)$. The maps are concocted so that they come with an inverse $\cP^H\to\cP$ and have the appropriate composition properties to define an $A_\infty$ homotopy equivalence.

For $0\leq t_1<t_2<t_3\leq1$, define $\ba=\Gamma^{t_3}$, $\bb=\Gamma^{t_2}$ and $\bb'=\Gamma^{t_1}$. Define $\bb^H$ as the result of sliding $\beta_1$ over $\beta_n$ (and possibly also another $\beta_{n'}$, disjoint from $\beta_n$), perturbed to be transverse to $\alpha_i,\beta_i,\beta_i'$ (see Figure~\ref{circles}). As always, assume all the perturbations are small enough so that there are obvious bijections between representatives of generators in $(\Sigma,\ba,\bb^H)$, $(\Sigma,\bb,\bb^H)$ and $(\Sigma,\bb^H,\bb')$. For each of these triples, the perturbations can be chosen so that there are obvious domains which are analogous to $D_i,E_i,T_i,D_i^H$ and $E_i^H$ in \cite[Figure~11]{L} (and also generators $\mathbf{\theta}_{\beta\beta'}=\{\theta_1,\ldots,\theta_k\}$, $\mathbf{\theta}_{\beta^H\beta'}=\{\theta'_1,\ldots,\theta'_k\}$ and $\mathbf{\theta}_{\beta\beta^H}=\{\theta^H_1,\ldots,\theta^H_k\}$ coming from their corners). Note that all of these domains can be made to simultaneously have arbitrarily small area by choosing sufficiently small perturbations and by appropriately performing the slide. \begin{figure}\capstart\begin{center}\includegraphics[width=\linewidth]{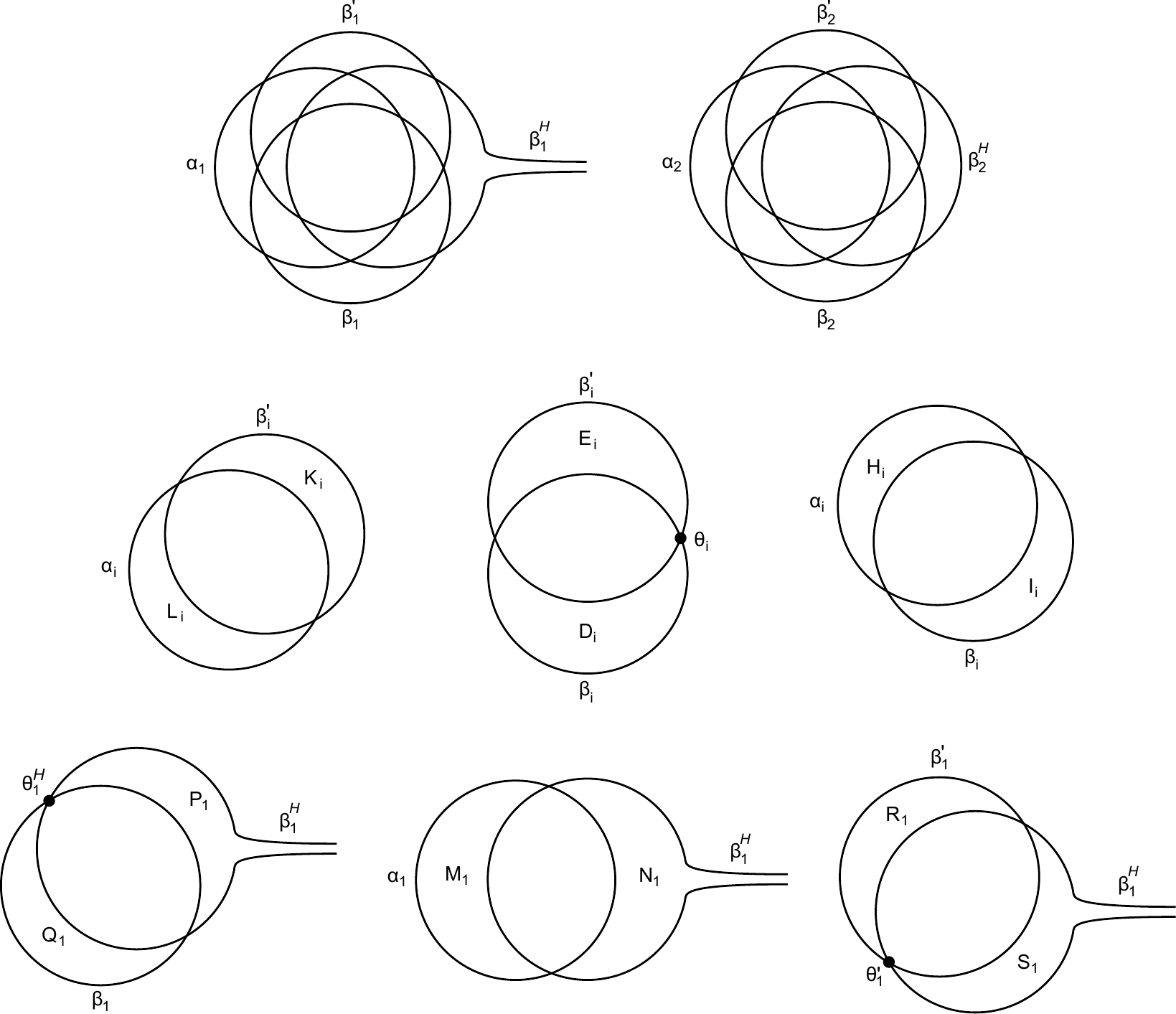}\end{center}\caption{Circles and domains mentioned in the proof of Proposition~\ref{lil11.2}. For $i>1$ the domains and circles are like those for $i=2$. The domains $P_i,N_i,S_i$ are analogous to the domains $E_i^H$ in \cite[Figure~11]{L}. There are analogous pictures for a slide of $\alpha_1$ over $\alpha_n$.}\label{circles}\end{figure} 

The maps below are inspired by the Heegaard-Floer triangle maps, though they count a restricted class of curves (those with domains mentioned above) that in the three-dimensional theory turn out to be the only classes to count, at least for handleslide maps. For sets of circles $\boldsymbol{a}_1,\ldots,\boldsymbol{a}_n$, let $\Delta_n$ denote the unit disk in $\C$ with $n$ boundary punctures and define 
\[W_{\boldsymbol{a}_1,\ldots,\boldsymbol{a}_n}=:(\Sigma\times\Delta_n,\{\boldsymbol{a}_i\times e_i:i=1,\ldots,n\}),\] 
where $e_i$ denotes the $i^{th}$ edge of $\Delta_n$ as in Figure~\ref{polygon}.

Equip $W_{\bb'\bb^H\bb\ba}$ with a family of almost complex structures satisfying conditions (1)-(7) in \cite[Section~10.6.2]{L}. As discussed there, it is known that $W_{\bb'\bb^H\bb\ba}$ can be viewed as one element of a one-dimensional moduli space of squares, and that the two ends of this moduli space correspond to degenerations into the pairs $(W_{\bb^H\bb\ba},W_{\bb'\bb^H\ba})$ and $(W_{\bb'\bb^H\bb}, W_{\bb'\bb\ba})$. For each of these triangles there is an almost complex structure satisfying conditions (\textbf{J}$'\boldsymbol{1}$)-(\textbf{J}$'\boldsymbol{4}$) in \cite[Section~10.2]{L} coming from the degeneration of $W_{\bb'\bb^H\bb\ba}$. For the square and all triangles, the definitions and existence of admissible triple and quadruple diagrams could be realized analogously to those in \cite{L}, but they are irrelevant to the maps defined in this section because of the restrictions already placed on the homology classes of triangles and squares under consideration.
\begin{df}\label{thin}Let $\mathfrak{T}(\x,\y)$ (resp., $\mathfrak{T}(\x,\y,\z)$ and $\mathfrak{T}(\x,\y,\z,\w)$) denote the Abelian group generated by the \emph{thin} homology classes: those elements of  $\pi_2(\x,\y)$ (resp., $\pi_2(\x,\y,\z)$ or $\pi_2(\x,\y,\z,\w)$) whose support is contained in the regions labeled in Figure~\ref{circles}  (this includes the analogous domains $E'_i,D_i'$ etc. if $\beta_1$ also slides over $\beta_{n'}$). The parts labeled $P_1,N_1,S_1$ are  annuli analogous to $E_1^H$ in \cite[Figure~11]{L}.\end{df}
Because of the way these domains are defined, it is possible to make their total area arbitrarily small in comparison to each of the remaining domains simultaneously. For this reason, $\fm_n$ decomposes into a thin part and a thick part, $\fm_n=\fm_n^{thin}+\fm_n^{thick}$, where $\fm_n^{thin}$ counts curves in thin classes in $\pi_2(\arx,\y)$ for each $\y$, while $\fm^{thick}_n$ counts curves in $\pi_2(\arx,\y)\setminus\mathfrak{T}(\arx,\y)$. There is a similar decomposition $\fm_n^H=\fm_n^{H,thin}+\fm_n^{H,thick}$. Subscripts will appear when it is necessary to refer to the thin classes in a particular space. 

For various $k$-tuples of circles $\boldsymbol{a},\boldsymbol{b},\boldsymbol{c}$ chosen from $\ba,\ \bb,\ \bb^H$ and $\bb'$, let $\cP(\Sigma,\boldsymbol{a},\boldsymbol{b})$ denote the $\Z_2$ vector space of $k$-tuples of intersection points between circles, analogous to the definition of generator representatives for $\cP$ given a surface diagram. There are maps 
\[\varphi_{\boldsymbol{a}\boldsymbol{b}\boldsymbol{c}}\co\cP(\Sigma,\boldsymbol{a},\boldsymbol{b})\times\cP(\Sigma,\boldsymbol{b},\boldsymbol{c})\to\cP(\Sigma,\boldsymbol{a},\boldsymbol{c})\]
defined on pairs of $k$-tuples $(\x,\y)$ by the formula
\[\varphi_{\boldsymbol{a}\boldsymbol{b}\boldsymbol{c}}(\x,\y	)=\sum\limits_\z\sum\limits_{\substack{A\in\mathfrak{T}_{abc}(\x,\y,\z)\\\ind A=0}}\left(\#\mathcal{M}^A\right)\z.\]
The moduli spaces $\cM^A$ appearing in this construction are completely analogous to those used in the definition of $\fm_2$: we require $\Sigma\times\Delta_3$ to have an admissible almost complex structure, use Lagrangians $\boldsymbol{a}\times e^{1/3}$, $\boldsymbol{b}\times e^{2/3}$, $\boldsymbol{c}\times e^{3/3}$, and the moduli spaces $\cM^A$ are required to satisfy conditions (\hyperlink{m0}{{\bf M0}})-(\hyperlink{m6}{{\bf M6}}). It is necessary to compose two of these maps to get a map $\varphi\co\cP\to\cP^H$:
\begin{equation}\label{eq:varphi}\varphi(\x)=\varphi_{\bb^H\ba\ba^H}\left(\varphi_{\bb^H\bb\ba}(\bth_{\bb^H\bb},\x),\bth_{\ba\ba^H}\right)\end{equation}
This is the composition of triangle maps corresponding to sliding $\beta_1$ over $\beta_n$, then sliding $\alpha_1$ over $\alpha_n$: the input for a curve counted by $\varphi$ is a $k$-tuple of intersection points between $\ba$ and $\bb$, and the output is a $k$-tuple of intersection points between $\ba^H$ and $\bb^H$. The slide maps $\ff_n\co\cP^{\times n}\to\cP^H$ are built similarly to the isotopy map in Section~\ref{isotopy}, using maps $\varphi$ in place of interpolation maps:
\begin{equation}\label{eq:slidemap}\ff_n=\sum\limits_{i+j=n+1}\sum\limits^{n-j+1}_{\ell=1}\fm_i^H\left(\varphi(\x_1),\ldots,\varphi(\x_{\ell-1}),\varphi\left(\fm_j\left(\x_\ell,\ldots,\x_{\ell+j-1}\right)\right),\varphi(\x_{\ell+j}),\ldots,\varphi(\x_n)\right).\end{equation}
In this sum, the parameters $t_1$, $t_2$ defining the cylinders $\ba,\bb$ (and in turn the cylinders $\ba^H,\bb^H$) depend on what corner of $W_{n+1}$ or $W^H_{n+1}$ takes part in the concatenation: to streamline notation we assume $t_1$ and $t_2$ are chosen for each triangle in such a way that all cylinders meet smoothly. Here is another map that turns out to be a homotopy inverse for $\ff$: define $\tilde{\varphi}\co\cP^H\to\cP$,
\[\tilde{\varphi}(\x)=\varphi_{\bb\bb^H\ba}\left(\bth_{\bb\bb^H},\varphi_{\bb^H\ba^H\ba}\left(\x,\bth_{\ba^H\ba}\right)\right),\]
then define $\tilde{\ff}_n\co(\cP^H)^{\times n}\to\cP$ by replacing $\varphi$ with $\tilde\varphi$ and switching $\fm,\fm^H$ in Equation~\ref{eq:slidemap}. See Figure~\ref{varphi} for an illustration of the triangles used in the definitions of $\ff$ and $\tilde\ff$.
\begin{figure}[h]\capstart\begin{center}\includegraphics{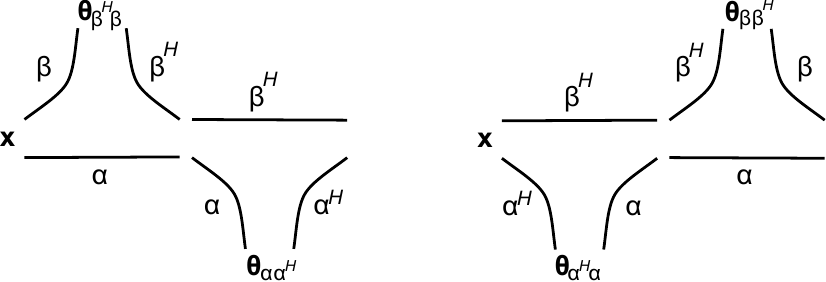}\end{center}\caption{Triangles used to define the map $\varphi(\x)$ at left and $\tilde{\varphi}(\x)$ at right, labeled with Lagrangians on the sides and limiting behavior at corners. The outputs are at the unlabeled corners. These are the spaces that replace interpolation strips in the definition of the slide maps.\label{varphi}}\end{figure}
\begin{lemma}\label{chainmaps}\ \begin{enumerate}
\item The slide maps $\ff,\tilde\ff$ are finite sums when $(\Sigma,\Gamma)$ is admissible.
\item The slide maps $\ff,\tilde\ff$ are $A_\infty$ morphisms.\end{enumerate}\end{lemma}
\begin{proof}We give the details for $\ff$; the corresponding statements for $\tilde\ff$ follow from the same arguments. The proof of the finiteness claim is just as in Lemma~\ref{finitesums}: any distinct elements $A,A'\in\mathfrak{T}_{\alpha^H\alpha\beta^H}(\x,\y,\z)$ or $\mathfrak{T}_{\bb^H\bb\ba}(\x,\y,\z)$ will have a difference of domains with positive and negative coefficients if $(\Sigma,\Gamma)$ is admissible, and the counts in the other polygons are finite by Lemma~\ref{finitesums} and the fact that a slide does not destroy admissibility.

It remains to show $\ff$ is an $A_\infty$ morphism. An outline of the argument is that $\boldsymbol{\theta}_{\beta\beta^H}$ and $\boldsymbol{\theta}_{\alpha\alpha^H}$ are elements of $\ker\fm_1^{thin}$ in a certain sense (Claim~\ref{thetacycles}), enabling the square used in the definition of $\varphi$ to substitute for an interpolation strip when only counting curves in thin classes. It may add perspective to note that for the handleslide maps of Heegaard-Floer theory, the only region in the diagram that is necessarily non-thin has the base point. Then the remaining coefficients in Equation~\ref{eq:morphismrlns} from thick curves cancel in a natural way (Claim~\ref{cancelthick1}).

Let $\fm_1^{thin,\beta\beta'}$ and $\fm_1^{thin,\beta\beta^H}$ denote the analogues of $\fm_1$ using the cylinders $\bb,\bb'$ and $\bb,\bb^H$ in $\Sigma\times\Delta_2$, but only counting holomorphic curves in thin classes.
\begin{claim}\label{thetacycles}The $k$-tuples of intersection points $\boldsymbol{\theta}_{\beta\beta'}$ and $\boldsymbol{\theta}_{\beta\beta^H}$ are in $\ker\fm_1^{thin,\beta\beta'}$ and $\ker\fm_1^{thin,\beta,\beta^H}$, respectively (and similar for $\bth_{\ba^H\ba'},\bth_{\ba\ba^H}$, and $\bth_{\bb^H\bb}$).\end{claim} 
\begin{proof}[Proof of Claim~\ref{thetacycles}]The first step is to explain why the index 1 classes counted in $\fm_1^{thin,\beta\beta^H}(\boldsymbol{\theta}_{\beta\beta^H})$ are precisely $D_1^H,\ldots,D_k^H,E_1^H,\ldots,E_k^H,E_1^H+E_2^H$ and $E_1^H+D_n^H$. First observe that any thin domain $\varphi$ (other than $E_1^H+E_n^H$ and $E_1^H+D_n^H$, treated separately) projects to a union of arcs in the $\beta_i$ under the obvious projection to $\beta_i$: the interiors of $E_i^H$ and $D_i^H$ project to disjoint arcs in $\beta_i$ whose common endpoints are in $\beta_i\setminus\cup_{j\neq i}\beta_j$. If $\varphi$ projects to an arc that runs along more than one $\beta$-circle before it ends, then its projection must turn right or left at some intersection between $\beta$-circles. But it is straightforward to use Proposition~\ref{indexprop} to see that this adds +1 to the index for each turn, and the thin domains with zero turns, namely the ones listed above, have index 1. It is also known that for $i>1$ the bigons $E_i^H$ and $D_i^H$ make canceling contributions to the coefficient of $\eta_i^h$ in $\partial_0(\boldsymbol{\theta}_{\beta\beta^H}$. Now suppose $\varphi$ is a thin domain other than $D_1^H,\ldots,D_k^H,E_1^H,\ldots,E_k^H$, with $n_{\theta_1^H}(C)=1/4$ (for example, $E_1^H+E_n^H$ or $E_1^H+D_n^H$). Traveling away from $\theta_1^H$ along the $\beta_1$ or $\beta_1^H$ boundary component of $C$, it is clear that one cannot reach a corner before reaching $\eta_1^H$ without violating the condition of being thin, so that $C$ must be either $E_1^H+E_2^H$ or $E_1^H+D_n^H$, and similar arguments to those in \cite{L} (starting with Sublemma~11.5) that exactly one of $\widehat{\mathcal{M}}^{E_1^H+E_2^H}$ and $\widehat{\mathcal{M}}^{E_1^H+D_n^H}$ contributes $+1$ and the other 0 to the generator $(\eta_1^H,\theta_2^H,\ldots,\theta_k^H)$). As with the other domains, this cancels with the contribution of $D_1^H$ for an appropriate choice of orientation system. The argument for $\partial_0(\boldsymbol{\theta}_{\beta\beta'})=0$ is the same as for $\boldsymbol{\theta}_{\beta\beta^H}$, except the case for $i=1$ is addressed the same as for $i>1$.\end{proof}
Claim~\ref{thetacycles} (and the obvious analogue for the pairs $\ba,\ba'$ and $\ba,\ba^H$) implies $\ff$ is an $A_\infty$ morphism with respect to $\fm^{thin},\fm^{H,thin}$ because the part of it concerning $\bth_{\ba\ba^H}$ and $\bth_{\ba^H\ba}$ implies the contributions to Equation~\ref{eq:morphismrlns} coming from the boundary of the relevant one-dimensional moduli spaces are restricted to those one would see if the squares corresponding to the inserted maps $\phi$ had no ends but those that took part in the concatenations. It remains to verify Equation~\ref{eq:morphismrlns} for thick domains.

We go through the argument in detail for $n=1$; with that understood, the general case is not hard to understand. For $n=1$, Equation~\ref{eq:morphismrlns} reads
\[\ff_2(\x,\fm_0)+\ff_2(\fm_0,\x)+\ff_1(\fm_1(\x))=\fm^H_1(\ff_1(\x)).\]
The coefficient of the generator $\w$ in the left side is a count of height two holomorphic buildings of three kinds, coming from the three degenerations of index one maps into $W_{\alpha\beta}$:\begin{itemize}
\item $\ff_2(\x,\fm_0)$ is a count of degenerations according to an arc in $W_{\alpha\beta}$ whose endpoints are in $\Sigma\times\{0\}\times\R$,
\item $\ff_2(\fm_0,\x)$ is a count of degenerations according to an arc in $W_{\alpha\beta}$ whose endpoints are in $\Sigma\times\{1\}\times\R$, and
\item $\ff_1(\fm_1(\x))$ is a count of degenerations according to an arc with one endpoint on each side of $W_{\alpha\beta}$. These consist of a thick index one strip in $\pi_2(\w',\w)$ and a thin square counted by $\ff_1$. (call these \emph{left degenerations}).\end{itemize}
The terminology comes from the idea of drawing the manifolds used in the definition of $\ff$ like in Figure~\ref{isotex} with all inputs at the left and all outputs at the right. In this case, the left side is everything whose output is concatenated with a manifold corresponding to $\varphi$, and the right side is everything with an input concatenated to the output of a manifold corresponding to $\varphi$. The first two types of degenerations cancel because $\fm_0$ is unchanged by the arbitrarily small perturbation sending $\ba$ to $\bb$. The coefficient of $\w$ coming from the right hand side is a count of height two holomorphic buildings consisting of a thin square counted by $\ff_1$ concatenated at its output with a thick index one strip in $W_{\alpha^H\beta^H}$ (call these \emph{right degenerations}).
\begin{figure}\capstart\begin{center}\includegraphics[width=\linewidth]{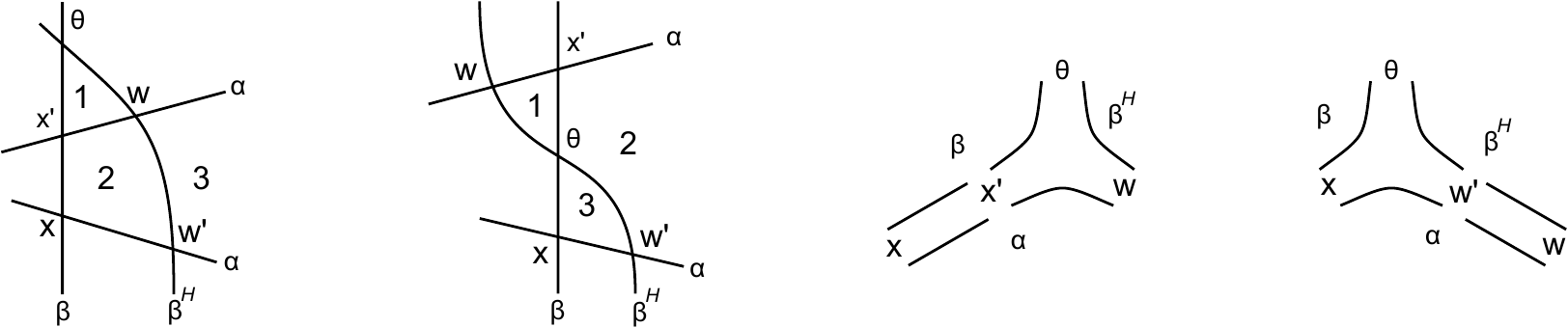}\end{center}\caption{The correspondence between domains of left and right degenerations can be given by considering what happens near the domains of the triangles, denoted with parentheses. The correspondence $\Phi$ between strips sends $(1)+2+3$ to $(1+2)+3$ in the first case, while it sends $(1)+2+3$ to $1+2+(3)$ in the second. The triangle $W_{\bb^H\bb\ba}$ is depicted with strips and boundary conditions labeled for readers' convenience.\label{corrfig}}\end{figure}
\begin{claim}\label{cancelthick1}There is a correspondence of right and left degenerations in thick classes such that $\ff_1(\fm_1(\x))+\fm^H_1(\ff_1(\x))=0$.
\end{claim}
\begin{proof}After an application of Lemma~\ref{isotlem}, without loss of generality assume each slide occurs along an embedded arc $\delta\co[0,1]\rightarrow\Sigma$, connecting $\beta_1$ to $\beta_n$, such that $\delta$ is disjoint from $\ba,\bb,\bb'$ and $\bb^H$ except at its endpoints. There is a linear bijection \[\Phi\co\pi_2^{\alpha\beta}(\x,\y)\to\pi_2^{\alpha\beta^H}(f_{\bb^H\bb\ba}(\x),f_{\bb^H\bb\ba}(\y))\] which appears as in \cite[Figure~41]{OSS} near the pair of pants given by the slide, adds or subtracts bigons in $\mathfrak{T}_{\beta\beta^H}$ to switch from $\bb$ boundary components to $\bb^H$ boundary components, and otherwise appears as in Figure~\ref{corrfig}. According to the same argument as for \cite[Lemma~7.21]{OSS}, $\ind{D}=\ind(\Phi(D))$ using the index formula.

The last step is to show that, for any classes $[u_L]$ with domain $D$ and $[u_r]$ with domain $\Phi(D)$ for left and right degenerations,\begin{equation}\#\mathcal{M}^{[u_L]}=\#\mathcal{M}^{[u_R]}.\label{correq}\end{equation}
Note that the image of $\Phi$ gives a complete account of the classes that could have representatives given by right degenerations, because any domain not in the image of $\Phi$ necessarily has a corner of the form $f_i$ as in \cite[Figure~41]{OSS}, but such corners are not corners of thin triangles, so such generators are not in the image of $f_{\bb^H\bb\ba}$. For the proof of Claim~\ref{cancelthick1} it remains to establish Equation~\ref{correq} using a neck-stretching argument.

\begin{figure}\capstart\begin{center}\includegraphics[width=0.8\linewidth]{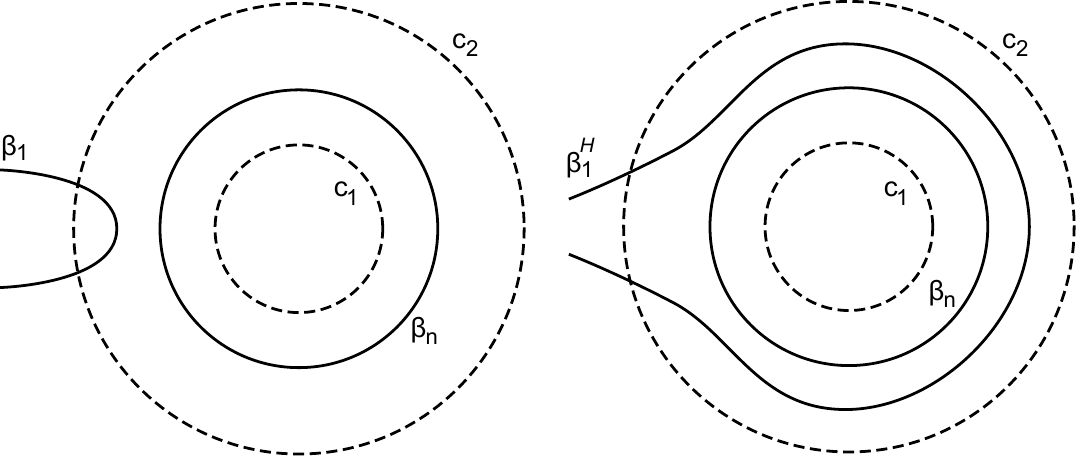}\end{center}\caption{The dotted circles give the location of necks inserted into various versions of $\Sigma\times\Delta_n$.\label{neck}}\end{figure}
It is possible to choose the two dotted circles $c_1,c_2$ in Figure~\ref{neck} close enough to $\beta_n$ so that all other circles in $\ba,\bb,\bb'$ and $\bb^H$ appear (if at all) as radial arcs within the annulus bounded by $c_1$ and $c_2$, with the exception of $\beta_1$ and $\beta_1^H$ as pictured there. Using the construction appearing \cite[Appendix A.2]{L} and gluing Propositions~A.1 and A.2 from the same paper, insert a neck of length $R$ (the two cylindrical manifolds inserted at each circle have length $R$) at $(c_1\cup c_2)\times[0,1]\times\R$ in both $W_{\alpha\beta}$ and $W_{\alpha\beta^H}$, and let $R\to\infty$. The limit spaces are denoted $W_{\alpha\beta}^\infty$ and $W_{\alpha\beta^H}^\infty$, both given by $\Sigma^\infty\times[0,1]\times\R$, where $\Sigma^\infty$ is a closed, orientable genus $g-1$ surface identified with a sphere at two points. For an index 1 homology class $A\in\pi_2^{\alpha\beta}(\x,\x')$, the corresponding homology class in $W_{\alpha\beta}^\infty$ breaks into two pieces that can be treated separately, one with representatives mapping to $\Sigma_{g-1}\times[0,1]\times\R$ and the other with representatives mapping to $S^2\times[0,1]\times\R$, both with marked points coming from intersections with the neck. The same is true for $\Phi(A)$ in $W_{\alpha\beta^H}^\infty$. The moduli spaces for $A$ and $\Phi(A)$ coming from the $\Sigma_{g-1}$ pieces are homeomorphic because the corresponding domains differ by an arbitrarily small perturbation of $\bb$ circles to $\bb^H$ circles within $\Sigma_{g-1}$ (this correspondence could be given by a flow in $W_{\alpha\beta^H}^\infty$ inherited from an arbitrarily small Hamiltonian $\Sigma_g\to\R$ that vanishes near the dotted circles, sending $\bb$ to $\bb'$). Finally, one may identify the moduli spaces for the pieces mapping to $S^2\times[0,1]\times\R$ using a rather indirect argument: the domains involved are exactly what would appear in the analogous neck-stretching construction for a Heegaard diagram (whose base point lies away from $\beta_n$) that undergoes a handleslide in Heegaard-Floer homology. The only possible exceptions are path components of curves whose domains are bounded entirely by radial $\beta$ circles, but these will be seen to have appropriate counts as well, since the counts in such classes are the same whether they are bounded by $\alpha$ circles, which could occur in Heegaard diagrams, or by $\beta$ circles. The existence of a homeomorphism between moduli spaces in that situation follows from the fact that the handleslide triangle map yields the chain map $CF^\infty(Y_{\alpha\beta})\to CF^\infty(Y_{\alpha\beta^H})$. This concludes the proof of Claim~\ref{cancelthick1}.\end{proof}
The two claims establish Equation~\ref{eq:morphismrlns} for $n=1$. For $n>1$, using the thin squares used to define $\varphi$ instead of interpolation strips, there are manifolds $W^p_n$ and moduli spaces $\cM_1^{n,p}(\arx,\y)$ analogous to those in Lemma~\ref{isotmorlem}. Now in addition to the left and right degenerations there are other boundary points, which we will also call left and right degenerations, coming from level splitting according to arcs that are disjoint from the squares used to define $\varphi$. Here, as before, the neck stretching argument (which was independent of the number of inputs) results in identical counts of holomorphic curves in homology classes corresponding under $\Phi$, and these give the count required by Equation~\ref{eq:morphismrlns}. This concludes the proof of Lemma~\ref{chainmaps}.\end{proof}
\begin{lemma}[c.f. Proposition~10.29 of \cite{L}]\label{chainh}The maps $\tilde{\ff}\circ\ff$ and $\ff\circ\tilde\ff$ are homotopic to the identity morphisms on $(\cP,\fm)$ and $(\cP^H,\fm^H)$, respectively, so that $\ff$ is a homotopy equivalence.
\end{lemma}
This lemma mostly follows from the following result.
\begin{sublemma}\label{sublem}The map $\tilde{\ff}\circ\ff$ is homotopic to the map that results from inserting 
\[\varphi_{\bb\bb^H\ba}\left(\bth_{\bb\bb^H},\varphi_{\bb^H\bb\ba}\left(\bth_{\bb^H\bb},\x\right)\right)\]
instead of interpolation maps as with previous morphisms.\end{sublemma}To clarify the intent of this sublemma, note that the to maps $\varphi$ appearing in the lemma correspond to the triangles that are at the far left and far right in Figure~\ref{varphi}. We are in effect claiming that the two inner triangles in some sense cancel each other.\begin{proof}First, note that in the composition there is a cancellation like in Lemma~\ref{isotmorcomp} that leaves precisely the manifolds that would be used in the definition of the morphism obtained by inserting the concatenation, from left to right, of the four triangles in Figure~\ref{varphi}. We now define the homotopy by interpreting the concatenation of the middle two triangles in that figure as a square, slightly perturbing the $\ba$ cylinders on the right side to be $\ba'$ cylinders to preserve admissibility (this is acceptable, because the perturbation induces a bijection of representatives of generators, so the resulting ``new'' maps, as maps on generators, are equal). More precisely, define the map
\[h\co\cP(\Sigma,\bb^H,\ba)\times\cP(\Sigma,\ba,\ba^H)\times\cP(\Sigma,\ba^H,\ba')\to\cP(\Sigma,\bb^H,\ba')\]
as a count of thin squares, given by the formula
\[h(\x,\y,\z)=\sum\limits_{\w}\sum\limits_{\substack{A\in\mathfrak{T}_{\bb^H\ba\ba^H\ba'}(\x,\y,\z,\w)\\ \ind A=-1}}\left(\#\mathcal{M}^A\right)\w,\]
Where the moduli space that appears is the union of the moduli spaces for a one-parameter family of almost complex structures on the square, analogous to the construction in \cite[Section 10.6.2]{L}. 

Consider the manifolds used in the isotopy map: For the addend
\begin{equation}\label{eq:addend}\fm_i\left(a_1,\ldots,a_{\ell-1},\fm_j\left(a_\ell,\ldots,a_{\ell+j-1}\right),a_{\ell+j},\ldots,a_n\right)\end{equation} 
in Equation~\ref{eq:ainftyrlns} there appears a $j$-tuple of interpolation strips, each one outputting an input for $\fm_j'$. We have been using them as placeholders for other manifolds; for example, in the slide morphism they were simultaneously replaced by the concatenation of two triangles. Enumerate these places $1,2,\ldots,j$ from top to bottom, so that, for example, the locations of the three strips in Figure~\ref{isotex} are labeled 1, 2, 3 going from top to bottom. Now for a given addend as in \ref{eq:addend}, instead of one manifold we define the collection $\{W_i\}_{i=1,\ldots,j}$ in which the element $W_i$ is constructed as follows.\begin{itemize}
\item The places indexed less than $i$ are given the concatenation of the triangles in Figure~\ref{varphi}, from left to right, that defines \[\varphi_{\bb\bb^H\ba'}\left(\bth_{\bb\bb^H},\varphi_{\bb^H\ba^H\ba'}\left(\varphi_{\bb^H\ba\ba^H}\left(\varphi_{\bb^H\bb\ba}\left(\bth_{\bb^H\bb},\x\right)\bth_{\ba\ba^H}\right)\bth_{\ba^H\ba'}\right)\right).\] (These are the manifolds that would be used to define $\tilde\ff\circ\ff$.)
\item Place $i$ gets the manifold that defines $\varphi_{\bb\bb^H\ba'}\circ h\circ\varphi_{\bb^H\bb\ba}$.
\item The places indexed greater than $i$ get the manifold that defines \[\varphi_{\bb\bb^H\ba}\left(\bth_{\bb\bb^H},\varphi_{\bb^H\bb\ba}\left(\bth_{\bb^H\bb},\x\right)\right)\]
(Recall these are the manifolds that would be used to define the other map mentioned in Sublemma~\ref{sublem})
\end{itemize}In this way, each addend \ref{eq:addend} of the $n^\text{th}$ $A_\infty$ relation contributes $j$ summands to the map
\[\fH_n\co\cP(\Sigma,\alpha,\beta)^{\times n}\to\cP(\Sigma,\alpha',\beta)\]
defined by summing over the collection $A(n)$ all addends in the $n^\text{th}$ $A_\infty$ relation, and over the collection $W(a)$ of manifolds defined for each addend $a\in A(n)$:
\begin{equation}\label{eq:homotopy}\fH_n(\arx)=\sum\limits_{a\in A(n)}\sum\limits_{W_i\in W(a)}\sum\limits_{\substack{A\in\mathfrak{T}_{W_i}(\arx,\y)\\\ind A=-1}}\#\left(\cM^A\right)\y.\end{equation}
We claim that this is the homotopy required by the sublemma. For those who care about gradings, note $\fH$ would have degree one less than the slide morphism $\ff$.

We now verify Equation~\ref{eq:htpyrlns} for $\fH$ as a homotopy from $\tilde\ff\circ\ff$ to the other map in Sublemma~\ref{sublem}, which we will call $\fg$. Recall there are two degenerations of the square. One corresponds to the break between the left two and right two triangles in Figure~\ref{varphi}. The ends of the moduli spaces counted in Equation~\ref{eq:homotopy} (except at index 0) coming from this degeneration contribute the term $\tilde\ff\circ\ff$. The other degeneration contributes the term $\fg$, because the inserted map corresponding to this degeneration is the identity map $\cP(\Sigma,\ba,\bb^H)\to\cP(\Sigma,\bb^H,\ba')$ on the nose: The fact that $\varphi_{\ba\ba^H\ba'}\left(\bth_{\ba\ba^H},\bth_{\ba^H\ba'}\right)=\bth_{\ba\ba'}$ follows just as in \cite[Proposition~11.3]{L}, so that 
\[\varphi_{\bb^H\ba\ba'}\left(\x,\varphi_{\ba\ba^H\ba'}\left(\bth_{\ba\ba^H},\bth_{\ba^H\ba'}\right)\right)=\varphi_{\bb^H\ba\ba'}\left(\x,\bth_{\ba\ba'}\right),\]
and $\varphi_{\bb^H\ba\ba'}\left(\x,\bth_{\ba\ba'}\right)$ counts the obvious unique holomorphic ``small triangle'' that exists for any $(\x,\boldsymbol{\theta}_{\ba\ba'})$, which outputs the unique representative in $\cP(\Sigma,\bb^H,\ba')$ of the same generator $\x$ (c.f. \cite[Proposition 9.8]{OS}). There is also the obvious correspondence of right and left degenerations in thick homology classes according to arcs in $W_i$ that are disjoint (that is, entirely to the right or to the left) from the inserted manifolds listed in its definition: the right degenerations contribute the second line of Equation~\ref{eq:htpyrlns} while the left degenerations contribute the last line.\end{proof}
\begin{proof}[Proof of Lemma~\ref{chainh}]By the same kind of argument as the proof of Sublemma~\ref{sublem}, The map that results from inserting 
\[\varphi_{\bb\bb^H\ba}\left(\bth_{\bb\bb^H},\varphi_{\bb^H\bb\ba}\left(\bth_{\bb^H\bb},\x\right)\right)\]
instead of interpolation maps as with previous morphisms is homotopic to the map one would get by inserting the triangle for $\varphi_{\bb'\bb\ba}$, which also induces the identity morphism because the triangle itself gives the identity map on generators. With an appropriate replacement of $\ba$ and $\bb$ symbols, the same argument shows the analogous result for $\ff\circ\tilde\ff$.\end{proof}
\end{section}

\ \\ \emph{University of Georgia, Athens\\ jdw.math@gmail.com}
\end{document}